\documentclass[a4paper,11pt]{amsart}

\usepackage{graphicx}
\usepackage{amsmath,amssymb,amsfonts}
\usepackage{mathrsfs}
\usepackage{amsthm}
\usepackage[numbers]{natbib}
\usepackage[T1]{fontenc}
\usepackage{newtxtext}
\usepackage{newtxmath}
\usepackage{textcomp}
\usepackage[colorlinks]{hyperref}

\usepackage{stmaryrd}
\usepackage{url}
\usepackage{fancyhdr}
\usepackage{mathtools}
\usepackage{color}
\usepackage{enumerate,mdwlist}
\usepackage{amstext}
\usepackage{array}

\usepackage[colorinlistoftodos, textwidth=1.2cm,textsize=small]{todonotes}

\usepackage{tikz}
\usetikzlibrary{matrix,positioning,arrows,decorations.pathmorphing}
\usetikzlibrary{tikzmark}

\usepackage[shortlabels]{enumitem}
\setlist[enumerate]{leftmargin=25pt}
\setlist[itemize]{leftmargin=25pt}

\newtheorem{thm}{Theorem}[section]
\newtheorem{lemma}[thm]{Lemma}
\newtheorem{prop}[thm]{Proposition}
\newtheorem{cor}[thm]{Corollary}

\newtheorem*{conj*}{Conjecture}

\newtheorem{introthm}{Theorem}
\newtheorem{introthmb}{Theorem}

\newtheorem{introconj}{Conjecture}
\newtheorem{introconjb}{Conjecture}

\theoremstyle{definition}
\newtheorem{rem}[thm]{Remark}

\newtheorem*{rem*}{Remark}

\theoremstyle{definition}
\newtheorem{defn}[thm]{Definition}

\newcommand{\emphd}[1]{{\color{blue} #1}}

\DeclareMathOperator{\eval}{eval}
\DeclareMathOperator{\dep}{dep}

\DeclareMathOperator{\vol}{vol}
\DeclareMathOperator{\length}{length}
\DeclareMathOperator{\lin}{lin}

\DeclareMathOperator{\conv}{conv}

\newcommand{\bra}[1]{{\left({#1}\right)}}
\newcommand{\floor}[1]{{\lfloor{#1}\rfloor}}
\newcommand{\ceil}[1]{{\lceil{#1}\rceil}}
\newcommand{\ent}[1]{{\left[{#1}\right]}}

\newcommand{\set}[1]{{\left\{{#1}\right\}}}

\DeclarePairedDelimiter\abs{\lvert}{\rvert}
\DeclarePairedDelimiter\norm{\lVert}{\rVert}

\newcommand{\Ga}{\ensuremath{\Gamma}}

\newcommand{\la}{\ensuremath{\lambda}}

\newcommand{\ZZ}{\ensuremath{\mathbb{Z}}}

\newcommand{\QQ}{\ensuremath{\mathbb{Q}}}
\newcommand{\RR}{\ensuremath{\mathbb{R}}}
\newcommand{\NN}{\ensuremath{\mathbb{N}}}

\newcommand{\sm}{\ensuremath{\setminus}}

\newcommand{\vn}{\ensuremath{\varnothing}}

\newcommand{\bs}[1]{\ensuremath{\mathbf{#1}}}

\definecolor{redi}{RGB}{255,38,0}
\definecolor{redii}{RGB}{200,50,30}
\definecolor{yellowi}{RGB}{255,251,0}
\definecolor{bluei}{RGB}{0,150,255}
\definecolor{blueii}{RGB}{135,247,210}
\definecolor{blueiii}{RGB}{91,205,250}
\definecolor{blueiv}{RGB}{115,244,253}
\definecolor{bluev}{RGB}{1,58,215}
\definecolor{orangei}{RGB}{240,143,50}
\definecolor{yellowii}{RGB}{222,247,100}
\definecolor{greeni}{RGB}{166,247,166}

\tikzset{ 
table/.style={
	matrix of nodes,
	row sep=-\pgflinewidth,
	column sep=-\pgflinewidth,
	nodes={rectangle,draw=black,text width=1.25ex,align=center},
	text depth=0.25ex,
	text height=1ex,
	nodes in empty cells
},
texto/.style={font=\footnotesize\sffamily},
title/.style={font=\small\sffamily}
}

\numberwithin{equation}{section}

\begin{document}

\title[Linearly-exponential checking for the Lonely Runner Conjecture]{Linearly-exponential checking is enough for the Lonely Runner Conjecture and some of its variants}

\author[R. Malikiosis]{Romanos Diogenes Malikiosis}
\address{Department of Mathematics\\
		 Aristotle University of Thessaloniki\\
         Thessaloniki\\
         Greece}
\email{rwmanos@gmail.com}

\author[F. Santos]{Francisco Santos}
\address{Departmento de Matem\'aticas, Estad\'istica y Computaci\'on\\
		  Universidad de Cantabria\\
         Santander\\
         Spain}
\email{santosf@unican.es}

\author[M. Schymura]{Matthias Schymura}
\address{University of Rostock\\
  Rostock\\
  Germany}
\email{matthias.schymura@uni-rostock.de}

\thanks{R. D. Malikiosis was granted a \emph{renewed research stay} by the Alexander von Humboldt Foundation for the completion of this project; R. D. Malikiosis also acknowledges that this project is carried out within the framework of the National Recovery and Resilience Plan Greece 2.0, funded by the European Union, NextGenerationEU (Implementation body: HFRI, Project Name: HANTADS, No. 14770).}
\thanks{Work of F. Santos is partially supported by grant PID2022-137283NB-C21
    funded by MCIN/AEI/10.13039/501100011033.}



\begin{abstract}
	Tao (2018) showed that in order to prove the Lonely Runner Conjecture (LRC) up to $n+1$ runners it suffices to consider positive integer velocities in the order of $n^{O(n^2)}$. Using the zonotopal reinterpretation of the conjecture due to the first and third authors (2017) we here drastically improve this result, showing that velocities up to $\binom{n+1}{2}^{n-1} \le n^{2n}$ are enough.
	
	We prove the same finite-checking result, with the same bound, for the more general \emph{shifted} Lonely Runner Conjecture (sLRC),
	except in this case our result depends on the solution of a question, that we dub the \emph{Lonely Vector Problem} (LVP), about sumsets of $n$ rational vectors in dimension two. We also prove the same finite-checking bound for a further generalization of sLRC that concerns cosimple zonotopes with $n$ generators, a class of lattice zonotopes that we introduce.
	
	In the last sections we look at dimensions two and three. In dimension two we prove our generalized version of sLRC (hence we reprove the sLRC for four runners), and in dimension three we show that to prove sLRC for five runners it suffices to look at velocities adding up to $195$.
\end{abstract}

\maketitle
		
\setcounter{tocdepth}{1}	
\tableofcontents


\section{Introduction and statement of results}

\subsection{The Lonely Runner Conjectures}

Let $\norm{x} = \min\{x - \floor{x}, \ceil{x} - x\}$ denote the distance from a real number~$x \in \RR$ to a nearest integer.
Dirichlet's approximation theorem states that for every $t\in\RR$ and $n\in \NN$, there is some $v\in\set{1,2,\dotsc,n}$
such that $\norm{vt}\leq\frac{1}{n+1}$.
A natural question is whether the set $\set{1,2,\dotsc,n}$ can be replaced by any other set of~$n$ distinct natural numbers, such that we obtain a better bound than~$\frac{1}{n+1}$. 
The claim that this question has a \emph{negative} answer is the content of the following famous problem:
 
 \begin{introconj}[Lonely Runner Conjecture]\label{lrc}
  Let $v_1,\dotsc,v_n$ be non-zero real numbers. Then, there is some $t\in\RR$ such that $\norm{v_it}\geq\frac{1}{n+1}$ for every $i$ with $1\leq i\leq n$.
 \end{introconj}

This was initially formulated by Wills in '68~\cite{willslrc}, and then an equivalent formulation as a geometric view-obstruction problem was given by Cusick in '73~\cite{cusickviewob}.
The name of the conjecture was given by Goddyn~\cite{NamingLRC}, who reinterpreted the problem in the following fashion:
Suppose that $n+1$ runners are running indefinitely on a circular track of length $1$ with constant, pairwise distinct, velocities and a common starting point; then the conjecture that each runner becomes ``lonely'' at some point in time, meaning that their distance to every other runner is at least $\frac{1}{n+1}$, is equivalent to 
Conjecture \ref{lrc}, by letting $v_1,\dots,v_n$ be the velocities of the other $n$ runners relative to the one we are looking at.

For a recent and comprehensive survey on the Lonely Runner Conjecture, we refer the reader to~\cite{perarnauserra2024thelonely}.
Following this survey and other sources on the problem, we refer to Conjecture~\ref{lrc} as ``LRC for $n+1$ runners'', since in Goddyn's interpretation there is an extra runner, which can be assumed to have velocity zero with no loss of generality.
LRC has been proven to hold for up to $7$ runners \cite{7lrc}.

Weaker versions of LRC have also been considered, where $\frac{1}{n+1}$ is replaced by a smaller fraction; the problem is trivially true if
we use $\frac{1}{2n}$, but it is still open if $\frac{1}{n+1}$ is replaced by $\frac{c}{n}$, for any $c>\frac{1}{2}$. The best effort in this direction is by Tao in '18 \cite{tao}, 
who showed that there is a time $t$ such that
\[
\norm{v_it}\geq\frac{1}{2n}+\frac{c\log n}{n^2(\log\log n)^2}, \;\;\;\text{ for all }i \text{ with }1\leq i\leq n,
\]
for some explicit absolute constant $c > 0$.

Usually the LRC is stated for \emph{distinct} velocities; this is no loss of generality, in the sense that if a counter-example to Conjecture~\ref{lrc} has a repeated velocity, then removing such a velocity provides a counter-example for smaller~$n$.
Distinctness of the velocities becomes crucial in certain variants of the conjecture, as we shall see further below.
We may also assume that all velocities~$v_i$ are positive \emph{integers}.
This reduction appeared first in german language as a combination of the arguments in the proofs of Wills~\cite[Lemma~5]{willslrc} and Betke \& Wills~\cite[Lemma~1)]{betkewills1972untere}.
Alternative proofs were given in~\cite{sixrunners} and~\cite{zonorunners}.
 
Excitingly, a much stronger reduction than to integer velocities was achieved by Tao~\cite{tao}.
In order to confirm the veracity of the LRC for up to $\leq n$ runners for a fixed integer $n \in \NN$, it suffices to check finitely many velocity vectors $(v_1,\dots,v_n) \in \ZZ_{>0}^n$.

\begin{thm}[\protect{\cite[Theorem 1.4]{tao}}]
\label{thm:tao-finite-checking-lrc}
There exists an absolute and explicitly computable constant $C > 0$, such that the following assertions are equivalent for every natural number $n \geq 1$:
\begin{enumerate}[(i)]
 \item The Lonely Runner Conjecture~\ref{lrc} holds for every number $m \leq n$ of velocities.
 \item The statement in the Lonely Runner Conjecture~\ref{lrc} holds
for all velocities $v_1,\ldots,v_m \in \ZZ_{>0}$ with $m \leq n$ and $v_i \leq m^{C m^2}$, for every $1 \leq i \leq m$.
\end{enumerate}
\end{thm}

\noindent Our first main result is an improvement of this bound from Tao's $n^{O(n^2)}$ to roughly~$n^{2n}$.
To state our bound, for each subset $S \subseteq [n]:=\{1,\dots,n\}$ let us denote $v_S:=\gcd(v_i: i \in S)$. 

\begin{introthm}\label{thm:main}
Suppose that LRC is true for $n$ runners.
Then, LRC is true for $n+1$ runners with integral positive velocities $0,v_1,\ldots,v_n$ satisfying $\gcd(v_1,\dotsc,v_n)=1$ and
\[
\sum_{S \subseteq [n]} v_S > \binom{n+1}{2}^{n-1}.
\]
\end{introthm}

We note that, independently of our paper, a similar result has been proved by Giri \& Kravitz~\cite{girikravitz2023structurelonelyrunnerspectra} with a bound of the order of $O(n^{\frac{5}{2}n})$.
The condition $\gcd(v_1,\dotsc,v_n)=1$ in Theorem~\ref{thm:main} is needed for our proof but is not relevant in practice: every minimal (with respect, e.g., to the sum of velocities) counter-example to LRC must automatically satisfy it, because simultaneously scaling a given set of velocities does not change the problem instance.

We are also interested in the \emphd{shifted} variant of the Lonely Runner Conjecture, which allows each runner to have an individual starting point~\cite{sLRC}.

\begin{introconj}[Shifted Lonely Runner Conjecture]\label{slrc}
Let $v_1,\dotsc,v_n$ be distinct non-zero real numbers. Then, for every $s_1,\dotsc,s_n\in\RR$ there is some $t\in\RR$ such that for every~$i$ with $1\leq i\leq n$ it holds $\norm{v_it+s_i}\geq\frac{1}{n+1}$.
\end{introconj}

This has been proved up to $n=3$ in~\cite{cslovjecsekmalikiosisnaszodischymura2022computing} and again in~\cite[Section~4.2]{fansun2023amending}.
As happened with Conjecture~\ref{lrc}, Conjecture~\ref{slrc} has also been reduced to the case where the velocities are positive integers (see~\cite[Section 4.1]{cslovjecsekmalikiosisnaszodischymura2022computing}).
However, in this conjecture one has to insist that all velocities be distinct
 because otherwise~$n$ runners with the same velocity $v$ and with equally spaced starting points (that is, $s_i = i/n$ for each $i$) would provide a counter-example:
at every time $t$ there is a runner~$i$ with $\norm{vt+i/n}\leq\frac{1}{2n}$. That the bound $\frac{1}{2n}$ is optimal for sLRC when equal velocities are allowed was proven by Schoenberg in an interpretation of the problem as billiard ball motions inside the unit cube~\cite{Schoenberg76} (see~\cite[Theorem~4]{sLRC} for an alternative proof).

In our attempt to prove a result similar to Theorem~\ref{thm:main} for sLRC we encountered an obstruction in the form
of a two-dimensional problem, which we dub the ``Lonely Vector Problem''.
Let $\bs P=\set{\bs p_1,\dotsc,\bs p_n}\subseteq \RR^2$ be a collection of~$n$ non-zero vectors, not all parallel and no two being equal or opposite to one another.
Let $S_{\bs P}$ be the following associated multiset of cardinality $n^2$:
\[
S_{\bs P} = \bs P \cup \set{\bs p_i+\bs p_j : 1\leq i<j\leq n} \cup \set{\bs p_i-\bs p_j : 1\leq i<j\leq n}.
\]

\begin{defn}
\label{defn:LVP}
We say that a set $\bs P$ as above satisfies the \emphd{Lonely Vector Property (LVP)} if the multiset~$S_{\bs P}$ contains a vector that is not parallel to any other vector in~$S_{\bs P}$.
We say that the Lonely Vector Property
holds for a certain~$n\in \NN$ if every such set of $n$ \emph{rational} vectors, not all parallel and no two of them equal or opposite to one another, satisfies the Lonely Vector Property.
\end{defn}

With this language our version of Theorem~\ref{thm:main} for the sLRC is as follows:

\begin{introthm}\label{thm:mainslrc}
Suppose that sLRC is true for $n$ runners  and that the LVP holds for every natural number $\leq n$. 
Then, sLRC is true for $n+1$ runners with integral positive \emph{distinct} velocities $0,v_1,\ldots,v_n$ satisfying $\gcd(v_1,\dotsc,v_n)=1$ and
\[
\sum_{S \subseteq [n]} v_S > \binom{n+1}{2}^{n-1}.
\]
\end{introthm}


\begin{rem}
\label{rem:ngon}
Consider the set $\bs P$ consist of vertices of a regular $m$-gon centered at the origin (taking only one from each pair of opposite vertices if $m=2n$ is even).
These vectors are not all parallel and no two equal or opposite, so we can ask ourselves whether they satisfy the LVP.
It turns out that they do not, unless the polygon is a triangle, square, or hexagon.
The fact that these three regular polygons are precisely the ones that can be linearly transformed to have \emph{rational} coordinates shows the importance of rationality in the LVP, and illustrates the difficulty of proving it; see Section~\ref{sec:LVP} for details.
\end{rem}

 In Section~\ref{sec:LVP}, we prove the LVP for up to $n=4$ points.
With this, Theorem~\ref{thm:mainslrc} implies that the sLRC holds for $n=4$  if it holds for all integer velocities with $v_1 + v_2 + v_3 + v_4 < 1000$. But we also  improve this bound to $195$ (see Theorem~\ref{thm:dim3} below).


\subsection{Zonotopal restatements of LRC and sLRC}
\label{subsec:zonotopes}

We here recall the reformulation of Conjectures~\ref{lrc} and \ref{slrc} in terms of zonotopes, developed in \cite{zonorunners}. 
We start with the following definition:

\begin{defn}
Let $Z = \sum_{i=1}^n\ent{\bs 0,\bs u_i} \subseteq \RR^{n-1}$ be a lattice zonotope in $\RR^{n-1}$ with~$n$ generators.
If the generators $\{\bs u_i : 1\leq i\leq n\}$ are in \emphd{linear general position}, that is, every $n-1$ of them are linearly independent,
we say that $Z$ is a \emphd{Lonely Runner Zonotope (LRZ)}.
\end{defn}

\begin{introconjb}[Equivalent to Conjecture~\ref{lrc}]\label{lrcgeom}
For every Lonely Runner Zonotope~$Z = \sum_{i=1}^n\ent{\bs 0,\bs u_i}$ of dimension $n-1$ we have
\begin{align}
\frac{n-1}{n+1}(Z-\bs x)\cap(\bs x+\ZZ^{n-1}) \neq \vn,
\label{eq:lrc}
\end{align}
where $\bs x=\frac{1}{2}(\bs u_1+\dotsb+\bs u_n)$ is the center of $Z$.
\end{introconjb}

Let us describe in short how this equivalence works. 
As mentioned in the previous section, there is no loss of generality in assuming that the velocities $v_1,\dotsc,v_n$ are integers and that $\gcd(v_1,\dots,v_n)=1$. We associate with such velocities a set of  vectors
\[\bs u_1,\dotsc,\bs u_n\in\ZZ^{n-1},\]
that span $\ZZ^{n-1}$ as a lattice and satisfy
\begin{equation}\label{eq:lindepui}
\begin{pmatrix}
   | & \ldots & |\\
   \bs u_1 & \ldots & \bs u_n\\
   | & \ldots & |
  \end{pmatrix}
  \begin{pmatrix}
   v_1\\
   v_2\\
   \vdots\\
   v_n
  \end{pmatrix}=\bs 0.
\end{equation}
We will denote by
\[
Z_{\bs v}:=\sum_{i=1}^n\ent{\bs 0,\bs u_i} \subseteq \RR^{n-1}
\]
the lonely runner zonotope associated in this way with the velocity vector
\[
\bs v = (v_1,\dotsc,v_n) \in \ZZ_{>0}^n.
\]
$Z_{\bs v}$ is unique modulo unimodular equivalence.
One way to construct it is to start with $\bs u'_i=-v_n \bs e_i$ for $i=1,\dots,n-1$ and $\bs u'_n=(v_1,\dots,v_{n-1})$, and then 
let $\bs u_i= T(\bs u'_i)$ where $T:\RR^{n-1}\to\RR^{n-1}$ is an arbitrary linear transformation sending a lattice basis of $\ZZ\bs u'_1 + \dotsc + \ZZ \bs u'_n$ to the standard basis.

Conversely, starting with a zonotope $Z = \sum_{i=1}^n\ent{\bs 0,\bs u_i}$ with vectors $\bs u_1, \dots , \bs u_n\in \ZZ^{n-1}$ in linear general position that generate the lattice $\ZZ^{n-1}$, let $v_1,\dots,v_n$ be the 
coefficients of the unique linear dependence among them, suitably scaled so that the $v_i$'s are integers and have no trivial common factor.
These coefficients satisfy \eqref{eq:lindepui} and are nonzero by general position of the $\bs u_i$'s. 
In fact, they are the minors of size $n-1$ of the matrix 
$\begin{pmatrix}
   \bs u_1 & \ldots & \bs u_n\\
\end{pmatrix}$;
this implies that
\begin{align}
\vol(Z_{\bs v}) = v_1 + \dotsc + v_n.\label{eqn:volZ-velocity-sum}
\end{align}
For this reason we call $\bs v$ the \emphd{volume vector} of $Z_{\bs v}$. 
See Section~\ref{subsec:lat-zonotopes} for more details.

The equivalence of  Conjectures~\ref{lrc} and~\ref{lrcgeom} follows then from the following more precise statement:

\begin{prop}[\cite{zonorunners}]
\label{prop:lrcgeom}
For each $\bs v=(v_1,\dots,v_n) \in \ZZ_{> 0}^n$ with $\gcd(v_1,\dots,v_n)=1$, a time $t$ as required in 
Conjecture~\ref{lrc} exists for ${\bs v}$ if and only if 
$Z_{\bs v}$ satisfies~\eqref{eq:lrc}.
\end{prop}

\begin{proof}[Sketch of proof]
	For a given velocity vector $\bs v$, Conjecture~\ref{lrc} holds if and only if the line $\RR\bs v$ spanned by $\bs v$ in $\RR^n$ intersects the lattice arrangement of cubes $\ZZ^n+[\frac{1}{n+1},\frac{n}{n+1}]^n$ or, equivalently, if the cube $[\frac{1}{n+1},\frac{n}{n+1}]^n$ intersects the lattice arrangement of lines $\ZZ^n+\RR\bs v$. Projecting orthogonally onto $H=\bs v^{\perp}$, the lattice arrangement $\ZZ^n+\RR\bs v$ collapses into a lattice $\Ga$, whereas the cube $[\frac{1}{n+1},\frac{n}{n+1}]^n$ becomes a zonotope, generated by $n$ vectors of $H$ in linear general position. Then, there exists a linear isomorphism between $H$ and $\RR^{n-1}$, which transforms said zonotope to $Z_{\bs v}$ and $\Ga$ to $\ZZ^{n-1}$. Hence, if Conjecture~\ref{lrc} holds for $\bs v$, then $Z_{\bs v}$ satisfies~\eqref{eq:lrc}. 
	
	It is not hard to show the converse as well. Suppose that $Z_{\bs v}$ satisfies~\eqref{eq:lrc}; $Z_{\bs v}$ is paved by $n$ parallelepipeds whose volumes are exactly the coordinates of $\bs v$. This fact implies that there is a linear isomorphism between $\RR^{n-1}$ and the hyperplane $H=\bs v^{\perp}$ in $\RR^n$, such that the above arguments may be reversed. See more details  in \cite{zonorunners}.
\end{proof}

The sum $\sum_{S \subseteq [n]} v_S$ in our Theorems~\ref{thm:main} and~\ref{thm:mainslrc} is nothing but the number of lattice points in the lattice zonotope $Z_{\bs v}$ (see Corollary~\ref{cor:lrc-points}).%
\footnote{Observe that for $n \geq 2$ 
\[
|Z_{\bs v} \cap \ZZ^{n-1}| = \sum_{S \subseteq [n]} v_S \ge v_1 + \dotsb + v_n + 2^n - n-1 > v_1 + \dotsb + v_n.
\]
In view of~\eqref{eqn:volZ-velocity-sum}, this implies that in Theorems~\ref{thm:main-points} and~\ref{thm:mainslrc-points} one can replace ``lattice points'' by ``volume'', although this gives weaker statements.}
Hence, Theorem~\ref{thm:main} is equivalent to the following:

\begin{introthmb}
\label{thm:main-points}
Let $n\in \NN$.
If Conjecture~\ref{lrcgeom} holds for $n-1$,  then no counter-example to Conjecture~\ref{lrcgeom} for~$n$ can contain more than $\binom{n+1}{2}^{n-1}$ lattice points.
\end{introthmb}

\begin{rem}
Our arguments in Section~\ref{sec:LRC} leading to Theorem~\ref{thm:main-points} are not dependent on the precise value~$\frac{n-1}{n+1}$ in Conjecture~\ref{lrcgeom}.
In fact, 
\begin{align}
 (1-2\eta)(Z-\bs x)\cap(\bs x+\ZZ^{n-1}) \neq \vn \label{eq:lrcrelaxed}
\end{align}
holds for every Lonely Runner Zonotope~$Z = \sum_{i=1}^n\ent{\bs 0,\bs u_i}$ of dimension $n-1$ for some $\eta>0$, if and only if the version of Conjecture~\ref{lrc} for $n$ where the bound $\frac{1}{n+1}$ is replaced by $\eta$ holds.
Hence, assuming Conjecture~\ref{lrcgeom} to hold in dimension $n-1$ with a constant $\frac{n-1}{n+1} + \varepsilon$, for some small $\varepsilon > 0$, results in a finite checking result (depending on~$\varepsilon$) for a weakened version of the LRC in dimension~$n$.
\end{rem}

We now look at the sLRC. Since it assumes different velocities, it involves the following class of zonotopes:

\begin{defn}
Let $Z_{\bs v} \subseteq \RR^{n-1}$ be a LRZ associated with the velocity vector $\bs v = (v_1,\dots,v_n) \in \ZZ_{>0}^n$.
When the entries of $\bs v $ are pairwise distinct we say that~$Z_{\bs v}$ is a \emphd{Strong Lonely Runner Zonotope (sLRZ)}.
\end{defn}

Recall that the covering radius $\mu(M)$ of a convex body $M \subseteq \RR^d$, is the smallest dilation factor $\rho>0$, such that $\rho M+\ZZ^d = \bigcup_{\bs z \in \ZZ^d} (\rho M + \bs z)$ covers the entire space~$\RR^d$ (see more details about it in Section~\ref{sec:prelim}). The reformulation of sLRC is as  follows:

\begin{introconjb}[Equivalent to Conjecture~\ref{slrc}]\label{slrcgeom}
Every Strong Lonely Runner Zonotope $Z_{\bs v}$ of dimension $n-1$ satisfies 
\begin{align}
\mu(Z_{\bs v}) \leq 
\frac{n-1}{n+1}.
\label{eq:slrc}
\end{align}
\end{introconjb}
\noindent Again the equivalence follows from the following more precise statement:

\begin{prop}[\cite{zonorunners}]
\label{prop:slrcgeom}
Fix $\bs v=(v_1,\dots,v_n) \in \ZZ_{>0}^n$ with pairwise distinct entries and $\gcd(v_1,\dots,v_n)=1$. A time $t$ as required in 
Conjecture~\ref{slrc} exists for every choice of $(s_1,\dots,s_n)$ if and only if $\mu(Z_{\bs v}) \leq \frac{n-1}{n+1}$.
\end{prop}

For the sake of comparison, we note that for each LRZ~$Z \subseteq \RR^{n-1}$ with center~$\bs x$, Equation~\eqref{eq:lrc} is equivalent to
\[
\frac{2n}{n+1}\bs x \in
\frac{n-1}{n+1}Z + \ZZ^{n-1},
\]
while Equation~\eqref{eq:slrc} is equivalent to the obviously stronger assertion
\[
\RR^{n-1} \subseteq \frac{n-1}{n+1}Z + \ZZ^{n-1}.
\]

In view of Proposition~\ref{prop:slrcgeom}, 
the following is equivalent to Theorem~\ref{thm:mainslrc}.

\begin{introthmb}
\label{thm:mainslrc-points}
Let $n\in \NN$.
If Conjecture~\ref{slrcgeom} holds for $n-1$ and the LVP holds for every natural number at most~$n$, then no counter-example to Conjecture~\ref{slrcgeom} for~$n$ can contain more than $\binom{n+1}{2}^{n-1}$ lattice points.
\end{introthmb}

The reason why we need the Lonely Vector Property to prove  Theorem~\ref{thm:mainslrc-points} can be traced down to the fact that the projection of an sLRZ is, in general, not an sLRZ. To circumvent this we introduce the following definition. 

\begin{defn}
\label{defn:cosimple}
A finite collection of lattice vectors spanning $\RR^d$ is called \emphd{cosimple}, if there is a linear dependence involving them and having coefficients that are all non-zero and with pairwise different absolute values.
A lattice zonotope that is generated by a cosimple collection of vectors is called a \emphd{cosimple zonotope}.
\end{defn}

In this definition, and in what follows, we switch to represent dimension by~$d$ (instead of $n-1$) in order to reserve $n$ for the number of generators of our zonotopes.
Since every sLRZ with $n$ generators is a cosimple $(n-1)$-zonotope, the following conjecture that we introduce in this paper generalizes the sLRC:

\begin{introconj}
\label{conj:cosimple}
Every cosimple zonotope with generators in~$\ZZ^d$ has covering radius upper bounded by $d/(d+2)$.
\end{introconj}

The advantage of this generalization is that, using that every projection of a cosimple zonotope is cosimple, 
in this generalized setting we can prove a finite checking result that does not need the LVP.
In fact, the proof of the following statement is much simpler that those of Theorems~\ref{thm:main-points} and~\ref{thm:mainslrc-points}:

\begin{introthm}
\label{thm:finite-checking-cosimple-conditional}
If Conjecture~\ref{conj:cosimple} holds in dimension $d-1$ for all cosimple zonotopes, then no counter-example to Conjecture \ref{conj:cosimple} in dimension $d$ can contain more than $\binom{d+2}{2}^d$ lattice points.
\end{introthm}

Since Conjecture~\ref{conj:cosimple} is stronger than Conjectures~\ref{lrcgeom} and~\ref{slrcgeom}, this implies the following:

\begin{cor}
\label{cor:finite-checking-cosimple-conditional}
If Conjecture~\ref{conj:cosimple} holds in dimension $d-1$ for all cosimple zonotopes, then no counter-example to Conjectures~\ref{slrcgeom} or~\ref{lrcgeom} in dimension~$d$ can have more than $\binom{d+2}{2}^d$ lattice points.
\end{cor}

Theorems~\ref{thm:main-points} (hence Theorem~\ref{thm:main}), Theorem~\ref{thm:mainslrc-points} (hence Theorem~\ref{thm:mainslrc}), and Theorem~\ref{thm:finite-checking-cosimple-conditional} are proved, respectively, in  Sections~\ref{sec:LRC}, \ref{sec:LVP} and~\ref{sec:cosimple}, after introducing in Section~\ref{sec:prelim} some concepts and results from zonotope theory and from the geometry of numbers that we need for the proofs.

In Section~\ref{sec:LVP} we additionally prove the LVP for $n\leq 4$, and in Section~\ref{sec:cosimple} we give some properties of cosimple zonotopes.
For instance, we show that every cosimple zonotope has lattice-width at least three and that the converse almost holds~(Corollary~\ref{cor:widthge3}). Here, the \emphd{lattice-width} of a convex body $C \subseteq \RR^d$ is the minimum length of a segment of the form~$f(C)$, the minimum being taken over all integer linear functionals~$f$, that is, linear functionals $f : \RR^d \to \RR$ with $f(\ZZ^d) \subseteq \ZZ$ (see Section~\ref{sec:prelim} for more details).

\subsection*{Small dimension}

In the final two Sections~\ref{sec:dim2} and~\ref{sec:dim3} we look at Conjectures~\ref{slrcgeom} and~\ref{conj:cosimple} in low dimension,
proving them for $d=2$ and providing a volume bound for $d=3$. In fact, both results are proved for lattice zonotopes of lattice-width at least three, a setting more general than that of sLR or even cosimple zonotopes.

\begin{introthm}[See more precise Theorem~\ref{thm:cosimple-dim2}]
\label{thm:dim2}
Every lattice $2$-zonotope~$Z$ of lattice-width at least three has $\mu(Z)\le \frac12$, except for a certain parallelogram $P_{2,5}$ of area $5$ and with $\mu(P_{2,5})=\frac35$.
\end{introthm}

\begin{introthm}[See more precise Theorem~\ref{thm:dim3-b}]
\label{thm:dim3}
Every lattice $3$-zonotope $Z$ of lattice-width at least three and volume at least $196$ has $\mu(Z) < \frac35$, except for parallelepipeds projecting onto $P_{2,5}$.
\end{introthm}

In both statements the proof starts by using known bounds relating the lattice-width, volume and covering radius of convex bodies in dimensions two and three (Lemmas~\ref{lemma:AverkovWagner} and~\ref{lemma:IglesiasSantos}). In dimension two the general bound reduces our problem to zonotopes of lattice-width at most two or volume at most eight, which are relatively easy to classify. In dimension three the general bound easily proves Theorem~\ref{thm:dim3} for zonotopes of lattice-width at least four (Section~\ref{subsec:width4+}), but some additional work and a careful case-study is required for those of lattice-width three (Section~\ref{subsec:width3}).

We believe these low-dimensional results are interesting for several reasons:

On the one hand, even if Conjecture~\ref{slrcgeom} is already known for dimension two~\cite{cslovjecsekmalikiosisnaszodischymura2022computing}, our much more detailed Theorem~\ref{thm:dim2} allows us to prove part (2) of the following statement in arbitrary dimension. 

\begin{introthm}
\label{thm:special1}
Suppose that sLRC holds for $n-1$. Then it also holds for velocity vectors $(v_1,\dots,v_n)$ with $\gcd(v_1,\dots,v_n)=1$ that satisfy one of the following conditions:
\begin{enumerate}
\item All but one of them have a non-trivial common factor (Proposition~\ref{prop:allbut1gcd}). 
\item All but two of them have a common factor other than perhaps two or four. 
Moreover, if $n\geq 7$ then $\gcd=4$ cannot occur either 
(Corollary~\ref{cor:small-gcd}).
\end{enumerate}
\end{introthm}

On the other hand, Theorem~\ref{thm:dim3} shows that all potential counter-examples to Conjecture~\ref{conj:cosimple} for $d=3$, hence also those to Conjecture~\ref{slrc} for $n=4$, have volume at most 195:

\begin{cor}\label{cor:mainslrc}
sLRC holds for $n=4$ for all velocity vectors $(v_1,v_2,v_3,v_4) \in \ZZ_{>0}^4$ satisfying
\[
\sum_i v_1 + v_2 + v_3 + v_4 \geq 196.
\]
\end{cor}

This result opens the door for a computational proof of the sLRC for five runners.
\begin{rem}
\label{rem:alcantara}
After a preprint of this manuscript appeared, such a computational proof was carried out successfully in~\cite{alcantaracriadosantos2025covering}, so that we now know that sLRC holds for $n=4$.
Certifying that the covering radius of the $2\, 133\, 561$ relevant sLR $3$-zonotopes is bounded by~$\frac35$ required the authors of~\cite{alcantaracriadosantos2025covering} to devise a new (and compared with~\cite{cslovjecsekmalikiosisnaszodischymura2022computing}, more efficient) algorithm for bounding covering radii of rational polytopes in general.

As a by-product of their computations,~\cite{alcantaracriadosantos2025covering} also shows that there are only three velocity vectors $(v_1,v_2,v_3,v_4)$ that are tight for the sLRC, meaning that for them no bound larger than $\tfrac1{5}$ works in Conjecture~\ref{slrc} (equivalently, their associated sLR zonotopes have covering radius exactly $\frac35$). These are the vectors $(1,2,3,4)$, $(1,3,4,7)$, and $(1,3,4,6)$. The first two were known to be tight for the original LRC (which obviously implies tightness for sLRC) but $(1,3,4,6)$ is tight only for sLRC.
\end{rem}



\section{Preliminaries}
\label{sec:prelim}

\subsection{Volume and lattice points in lattice zonotopes}
\label{subsec:lat-zonotopes}

We here recall some known facts about zonotopes in general and lonely runner zonotopes in particular.

Let $Z = \sum_{i=1}^n\ent{\bs 0,\bs u_i}\subseteq \RR^d$ be a zonotope of dimension $d$ with $n$ generators. For each subset $S\subseteq [n]$ let us denote $Z_S:= \sum_{i\in S} \ent{\bs 0,\bs u_i}$.
If $|S|=d$, then $Z_S$ is a parallelepiped of volume $\abs{\det(\bs u_i: i \in S)}$ (degenerating to volume zero if $\{\bs u_i: i \in S\}$ is linearly dependent). It is well-known (see, e.g., \cite[Eq.~(57)]{shephard1974combinatorial}) that
\begin{align}
\vol(Z) = \sum_{S \subseteq [n], |S|=d} \vol(Z_S).
\label{eq:volZ}
\end{align}

Suppose now that $\bs u_i \in \ZZ^d$ for every $i$. Then it still makes sense to define $\vol(Z_S)$
for a subset $S \subseteq [n]$ of size other than~$d$, as the \emph{relative} $|S|$-dimensional volume of  $Z_S$, normalized to the lattice $\ZZ^d \cap (\sum_{i\in S} \RR \bs u_i)$.
Observe that we insist on $\vol(Z_S)$ to denote an $|S|$-dimensional volume, so again $\vol(Z_S)=0$ if $\{\bs u_i: i \in S\}$ is linearly dependent; for example, if $|S| > d$.
With this convention, $\vol(Z_S)$ coincides with the $\gcd$ of all the $(|S|\times |S|)$-minors of the $d \times |S|$ matrix with columns $\{\bs u_i : i \in S\}$. 
It turns out that the sum of all these volumes for parallelepipeds of dimensions from 0 to $d$ equals the number of lattice points in~$Z$.

\begin{thm}[see \protect{\cite[Corollary~9.3 \& Theorem~9.9]{BeckRobins}}]
\label{thm:zono-points}
Let $Z=\sum_{i=1}^n\ent{\bs 0,\bs u_i}\subseteq \RR^d$ be a lattice zonotope. 
Then
\begin{align}
|Z \cap \ZZ^d| &=\sum_{S \subseteq [n]} \vol(Z_S).
\label{eq:numZ}
\end{align}
\end{thm}
\noindent Comparing~\eqref{eq:volZ} and~\eqref{eq:numZ} one observes that every lattice zonotope contains more lattice points than its volume.

We now specialize to lonely runner zonotopes; that is, we assume $d=n-1$. Then, $(\vol(Z_{S}))_{|S|= n-1}$ equals the volume vector $(v_1,\dots, v_n)$ of $Z$, and  Equation~\eqref{eq:volZ} specializes to~\eqref{eqn:volZ-velocity-sum}.
The assumption $\gcd(v_1,\dots,v_n)=1$ is equivalent to the lattice $\ZZ^{n-1}$ being generated by the vectors~$\bs u_1,\dotsc,\bs u_n$.

We can also give a zonotopal interpretation of the greatest common divisor of subsets of the velocities.
As we did in the Introduction, for a subset $S\subseteq[n]$ we denote $v_S:= \gcd (v_i : i  \in S)$. 

\begin{lemma}
\label{lemma:gcd-geom}
Let $(v_1,\dots,v_n)$ with $\gcd(v_1,\dots,v_n)=1$ be the volume vector of an LRZ.
Then, for each $S \subseteq [n]$, we have
\[
v_{ [n] \setminus S} = \vol(Z_S).
\]
\end{lemma}
\noindent Note the special case $v_\emptyset = \vol(Z_{[n]})=0$ of this equality.

\begin{proof}
One direction is easy: the volume of $Z_S$ divides the volume of all the full-dimensional parallelepipeds of which $Z_S$ is a face, and those are precisely the ones of volumes $v_i$, $i\not\in S$. Thus, $\vol(Z_S)$ divides 
$v_{ [n] \setminus S} =\gcd (v_i : i \not \in S)$.

For the other direction, suppose that we do not have equality.
That is, there is a non-trivial factor $r \geq 2$ such that 
$ \gcd (v_i : i \not \in S) = r \vol(Z_S)$.
%
Let $k= |S|$. Without loss of generality suppose that the generators of $Z_S$ are $\bs u_1,\dotsc,\bs u_k$ and that they span the linear subspace of the first $k$ coordinates (the latter is no loss of generality since it can be achieved by a unimodular transformation).
That is, writing $\bs u_i = (\bs x_i, \bs y_i)$ with $\bs x_i\in \ZZ^k$ and $\bs y_i \in \ZZ^{n-k}$,
the matrix of the vectors $\bs u_i$ has the following block structure:
\[
\bs U:=
\begin{pmatrix}
   | & \ldots & |\\
   \bs u_1 & \ldots & \bs u_n\\
   | & \ldots & |
  \end{pmatrix}
  =
\begin{pmatrix}
   \bs x_1 \ \ldots \ \bs x_k &|&   \bs x_{k+1} \ \ldots \ \bs x_n\\
   \hline
   \bs 0 \ \ \ldots \ \ \bs 0 &|&   \bs y_{k+1} \ \ldots \ \bs y_n\\
  \end{pmatrix}  
=
\begin{pmatrix}
\bs S &|& \bs X \\
\hline
\bs 0 &|& \bs Y \\
\end{pmatrix},
\]
where $\bs S$ is a square $k\times k$ matrix of determinant $\vol(Z_S)$ and $\bs Y$ is an $(n-k)\times (n-k)$ matrix whose maximal minors are all divisible by $r$.
The contradiction is that then all the maximal minors of $\bs U$ are divisible by $r$. 
For those obtained by removing one of the last $n-k$ columns this is obvious (in fact, those minors are the $v_i$, for $i \not\in S$); for those obtained by removing one of the first~$k$ columns it follows by Laplace multirow expansion along the first~$k$ rows.
\end{proof}

\begin{cor}
\label{cor:lrc-points}
Let $Z\subseteq \RR^{n-1}$ be an LRZ with volume vector $(v_1,\dots,v_n)$ and for each $S\subseteq [n]$ let $v_S:= \gcd(v_i \in S)$. Then,
\[
|Z \cap \ZZ^{n-1}| = \sum_{S\subseteq [n]}  v_S.
\]
\end{cor}

\begin{proof}
Combining Theorem~\ref{thm:zono-points} and Lemma~\ref{lemma:gcd-geom} we have:
\[
|Z \cap \ZZ^{n-1}| = \sum_{S\subseteq [n]} \vol(Z_S) = \sum_{S\subseteq [n]} v_S.\qedhere
\]
\end{proof}

\subsection{Concepts from the Geometry of Numbers}
The geometric ideas in our proofs rest on some fundamental principles in the Geometry of Numbers
describing how the volume of a convex body~$C \subseteq \RR^d$ relates to whether $C$ contains lattice points, that is, points from~$\ZZ^d$, or whether it can be translated to avoid containing them.
Specific variants of such principles can be conveniently formulated by three well-studied parameters that we now introduce.

The first parameter is the \emphd{first successive minimum} and concerns \emphd{$\bs 0$-symmetric} convex bodies $C$, that is, those that satisfy $C = -C$.
It is defined as
\[
\lambda_1(C) := \min\left\{ \lambda > 0 : \lambda C \cap \ZZ^d \neq \left\{\bs 0\right\} \right\},
\]
and the lattice points in $\ZZ^d  \cap \lambda_1(C) C) \setminus \{\bs 0\}$
are said to  \emphd{attain} the first successive minimum of~$C$.
This parameter was introduced in the seminal works of Minkowski, whose first fundamental theorem says:

\begin{thm}[See \protect{\cite[Section~22]{gruber2007convex}}]
\label{thm:minkowski-first}
Let $C \subseteq \RR^d$ be a $\bs 0$-symmetric convex body.
If $C$ does not contain a non-zero point of $\ZZ^d$ in its interior, then $\vol(C) \leq 2^d$.

Equivalently, for every convex body $C$ one has
\[
\lambda_1(C)^d \vol(C) \leq 2^d.
\]
\end{thm}

The following related result is a version of~\cite[Theorem~2.1]{betkehenkwills1993successive}.
Observe that we no longer assume~$C$ to be $\bs 0$-symmetric, but $C-C$ always is.
The parameter $\lambda_1(C-C)^{-1}$ equals the greatest lattice length of a segment contained in~$C$:

\begin{lemma}
\label{lemma:minkowski-first-points}
Let $C \subseteq \RR^d$ be a convex body and $\ell\in \NN$.
If $|C\cap \ZZ^d| >\ell^d$, then
\[
\lambda_1(C-C) \leq \frac{1}{\ell}.
\]
\end{lemma}

\begin{proof}
By the pigeon-hole principle. $|C\cap \ZZ^d| >\ell^d$ implies that  there are two distinct lattice points $\bs x,\bs y\in C\cap \ZZ^d$ in the same class modulo $\ell\ZZ^d$. Then, $\frac1\ell(\bs x-\bs y) \in \ZZ^d \cap \frac1\ell(C-C)$ implies $\lambda_1(C-C) \leq 1/\ell$.
\end{proof}

\begin{rem}
Suppose that $C$ itself is centrally symmetric, that is,  $C - \bs x = \bs x - C$, for some point $\bs x$, not necessarily the origin.
Then, $\vol(C-C) = 2^d \vol(C)$, so Minkowski's theorem (Theorem~\ref{thm:minkowski-first}) gives
\[
\lambda_1(C-C)  \leq 1 / \sqrt[d]{\vol(C)}.
\]
Since ``number of  lattice points'' can be considered a discrete analogue of ``volume'', Lemma~\ref{lemma:minkowski-first-points}, restricted to centrally symmetric bodies, is a discrete version of Theorem~\ref{thm:minkowski-first}. In fact, in the case of interest to us, $C$ will be a lattice zonotope so, as noted right after Theorem~\ref{thm:zono-points}, its number of lattice points is larger than its volume.
\end{rem}

The second parameter important for us is the \emphd{covering radius} of a convex body $C \subseteq \RR^d$, which appeared already in the zonotopal reformulation of the sLRC (Conjecture~\ref{slrcgeom}).
For convenience, we define it with respect to an arbitrary lattice $\Lambda \subseteq \RR^d$ as
\[
\mu(C,\Lambda) := \min\left\{ \mu > 0 : \mu C + \Lambda = \RR^d \right\}.
\]
That is, the covering radius is the minimal positive dilation factor~$\mu$ such that every point in~$\RR^d$ is contained in some lattice translation of~$\mu C$.
When $\Lambda= \ZZ^d$ and then we write~$\mu(C)$ instead of $\mu(C, \ZZ^d)$ for brevity.

It is easy to prove (cf.~\cite[Chapter 2, \S 13]{gruberlekkerkerker1987geometry}) that $\mu(C,\Lambda)$ coincides with the maximum dilation factor $\rho > 0$ such that  $\rho C$ has a hollow translate, where we say that a convex body $C'$ is \emphd{hollow} if its interior contains no lattice points.
For example,  $\mu(C,\Lambda) \geq 1$ is equivalent to~$C$ admitting a hollow translate.

The following result is fundamental in our proofs, since it allows us to induct on the dimension when $\lambda_1$ is sufficiently small (which will be guaranteed by Lemma~\ref{lemma:minkowski-first-points}).

\begin{prop}
\label{prop:CSS-sum-bound}
Let $C \subseteq \RR^d$ be a convex body containing the origin, and let $\pi : \RR^d \to \RR^{d-1}$ be the projection along a direction that attains the first successive minimum of $C-C$. Then,
\[
\mu(C) \leq \lambda_1(C-C) + \mu\left(\pi(C),\pi(\ZZ^d)\right).
\]
\end{prop}

\begin{proof}
This is a special case of~\cite[Lemma~2.1]{codenottisantosschymura2019the}, which reads as follows:
Let $\pi: \RR^d \to \RR^l$ be a rational linear projection, so that $\pi(\ZZ^d)$ is a lattice.
Let $Q = C \cap \pi^{-1}(\bs 0)$ and let $L = \pi^{-1}(\bs 0)$ be the linear subspace spanned by~$Q$.
Then,
\[
\mu(C) \leq \mu(Q,\ZZ^d \cap L) + \mu(\pi(C),\pi(\ZZ^d)).
\]
Now, specializing the dimension to $l = d-1$, and the direction of the projection~$\pi$ to be along some $\bs w \in \lambda_1(C-C)(C-C) \cap \ZZ^d \setminus \{ \bs 0\}$, we find that $Q = [\bs x,\bs y]$, for some $\bs x,\bs y \in C$, and $\bs w = \lambda_1(C-C) (\bs x - \bs y)$.
Since we further have $\ZZ^d \cap \pi^{-1}(\bs 0) = \ZZ \bs w$, this implies
\[
\lambda_1(C-C) = \mu\left(\lambda_1(C-C)^{-1} [\bs 0,\bs w],\ZZ \bs w\right) = \mu(Q,\ZZ^d \cap \pi^{-1}(\bs 0)),
\]
and the statement follows.
\end{proof}

\noindent The second bound for $\mu$ that we need is apparently new and formulates as follows:

\begin{prop}
\label{prop:max-bound}
Let $C \subseteq \RR^d$ be a convex body and let $\pi: \RR^d \to \RR^{d-s}$ be a rational linear projection.
Then,
\[
\mu(C) \leq \max\left\{\mu(\pi(C),\pi(\ZZ^d)) , \max_{\bs y \in \pi(C)}\mu(C \cap \pi^{-1}(\bs y)) \right\},
\]
where $\mu(C \cap \pi^{-1}(\bs y))$ is considered  with respect to the lattice $\ZZ^d \cap \pi^{-1}(\bs 0)$ translated to the affine subspace $\pi^{-1}(\bs y)$.
\end{prop}

\begin{proof}
For the sake of brevity we write 
\[
\mu_{\bs y} = \mu(C \cap \pi^{-1}(\bs y)) \quad\textrm{ and }\quad \mu_\pi = \mu(\pi(C),\pi(\ZZ^d)).
\]
Let $L= \pi^{-1}(\bs 0)$, so that $s=\dim L$.
After we dilate $C$ by the factor $\mu_\pi$, every $s$-dimensional affine plane~$H$ parallel to $L$ intersects some integer translation of $\mu_\pi C$.
Hence, $H$ contains an integer translation of $ \mu_\pi C \cap\pi^{-1}(\bs y) $ for some $\bs y \in \pi(C)$.

If $\mu_\pi \geq \mu_{\bs y}$, then, by $\RR^d = L + \pi(\RR^d)$, we get $\mu(C) \leq \mu_\pi$.
So, now we assume that $\mu_\pi < \mu_{\bs y}$.
The argument above implies that adding to $\mu_\pi C$ a dilation of~$C$ by a factor $\mu_{\bs y} - \mu_\pi$ will guarantee a complete covering of~$H$.
This implies that $\mu(C) \leq \mu_\pi + \max_{\bs y}\{\mu_{\bs y} - \mu_\pi\} = \max_{\bs y} \mu_{\bs y}$, finishing the proof.
\end{proof}

\begin{prop}\label{prop:allbut1gcd}
Let $n \geq 3$ and suppose that Conjecture~\ref{slrcgeom} (hence Conjecture~\ref{slrc}) holds for $n-1$ velocities.
Then, it holds for~$n$ velocities that satisfy $\gcd(v_1,\dots, v_{n})=1$ if~$n-1$ of them have a common non-trivial factor. Equivalently, if the corresponding sLRC zonotope has a non-primitive generator.
\end{prop}

\begin{proof}
That a generator being not primitive is equivalent to all but one of the~$v_i$'s having a common factor is the special case of Lemma~\ref{lemma:gcd-geom}, where the set $S$ consists just of a single element.

So, let $Z = \sum_{i=1}^n\ent{\bs 0,\bs u_i}\subseteq \RR^{n-1}$ be an sLRZ and suppose that one generator is not primitive, say $\bs u_{n}=k \bs e_d$, for some $k\in \ZZ_{\ge 2}$.
Let $\pi:\RR^{n-1}\to \RR^{n-2}$ be the projection along~$\bs u_n$.
Then, Proposition~\ref{prop:max-bound} combined with the inductive hypothesis gives
\[
\mu(Z) \le \max\left\{\mu(\pi(Z)), \mu([\bs 0,\bs u_n]) \right\} \le \max\left\{\frac{n-2}{n}, \frac1k \right\} < \frac{n-1}{n+1},
\]
since $k \geq 2$ and $n \geq 3$.
\end{proof}

\medskip

For our study of the low-dimensional cases of Conjecture~\ref{conj:cosimple} in Sections~\ref{sec:dim2} and~\ref{sec:dim3}, we need some results relating $\mu(C)$ with a third parameter, the \emphd{lattice-width} of a convex body~$C$, defined as follows. For each $\bs z \in \RR^d$ define
\[
w(C, \bs z) =  \max_{\bs x \in C} {\bs z}^\intercal \bs x - \min_{\bs x \in C} {\bs z}^\intercal \bs x,
\]
and 
\[
w(C) = \min_{\bs z \in \ZZ^d \setminus \{\bs 0\}} w(C, \bs z).
\]

That is, $w(C)$ is the minimal width of~$C$ in a lattice direction, where width in each direction $\bs z \in \ZZ^d$ is normalized to the distance between lattice hyperplanes orthogonal to $\bs z$.
An observation that we frequently use below is that if $C$ is a \emphd{lattice polytope}, meaning that~$C$ is a polytope all of its vertices belong to~$\ZZ^d$, then the lattice-width $w(C)$ is an integer.

The following results bound the volume of wide hollow bodies. Equivalently, they bound the volume of an arbitrary convex body $C$ in terms of the product $w(C)\mu(C)$ (see Corollary~\ref{cor:AverkovWagner} for the latter perspective).
The versions the we need, for dimensions two and three, are due to
Averkov \& Wagner~\cite{AverkovWagner} and to Iglesias \& Santos~\cite{IglesiasSantos}, respectively.

\begin{lemma}[\protect{\cite[Theorem 2.2]{AverkovWagner}}]
\label{lemma:AverkovWagner}
Let $w > 1$ and let $C \subseteq \RR^2$ be a hollow convex body of lattice-width at least~$w$. 
Then, $\vol(C) \le \frac{w^2}{2(w-1)}$.
\end{lemma}

\begin{lemma}[\protect{\cite[Theorem 2.1]{IglesiasSantos}}]
\label{lemma:IglesiasSantos}
Let $C \subseteq \RR^3$ be a hollow convex body of lattice-width $w\ge 2+2/\sqrt3$.
Then, $\vol(C) $ is bounded from above by
\begin{align*}
&\frac{8w^3}{(w-1)^3}, && \text{ if } w\ge \frac2{\sqrt3}(\sqrt{5}-1) + 1 \approx 2.427, \text{ and}\\
&\frac{3w^3}{4(w - (1+2/\sqrt3))}, && \text{ otherwise} .
\end{align*}
\end{lemma}


\section[Projections of Zonotopes, and a finite checking result for LRC]{Projections of Lonely Runner Zonotopes, and a finite checking result for LRC}
\label{sec:LRC}

In this section, we prove Theorem~\ref{thm:main-points} (and thus Theorem \ref{thm:main}).
First, we prove an independent result that will allow us to apply induction on the dimension.

\subsection{Projections of Lonely Runner Zonotopes}
\label{sec:projections-lrz}

Let $Z_{\bs v}\subseteq \RR^{n-1}$ be an LRZ, with generators $\bs u_1, \dots, \bs u_n$
and center $\bs x_{\bs v}:=\frac{1}{2}\sum_{i=1}^n\bs u_i$. Consider a linear projection
\begin{align*}
T: \ \RR^{n-1} &\to \RR^{n-2} \\
\bs u_i &\mapsto \bs y_i
\end{align*}
satisfying $T(\ZZ^{n-1}) = \ZZ^{n-2}$.
Let $Z=T(Z_{\bs v})$, which is itself a lattice zonotope,  
\[
Z = \sum_{i=1}^n\ent{\bs 0,\bs y_i},
\]
with its center given by $\bs y_{\bs v}:=T(\bs x_{\bs v}) = \frac12\left(\bs y_1 + \dotsc + \bs y_n\right)$.

What we want to show is that

\begin{thm}
\label{thm:LRZinside}
The zonotope $Z$ contains a lattice translation of a lonely runner zonotope $Z' \subseteq \RR^{n-2}$ with center $\bs z\equiv\bs y_{\bs v}\bmod\ZZ^{n-2}$.
\end{thm}

For the proof, let $\bs w$ be a \emphd{primitive} vector in the kernel of $T$, which means that its coordinates have no non-trivial common factor.
We need to know how the volumes of parallelepipeds spanned by $n-2$ vectors among $\bs y_1,\dotsc,\bs y_n$ compare to those spanned by
$n-2$ vectors among $\bs w,\bs u_1,\dotsc,\bs u_n$.
To this end, let $\bs u_1',\dotsc,\bs u_{n}'$ 
denote a relabelling of the $n$ vectors  $\bs u_1,\dotsc,\bs u_n$ and let $\bs y_1',\dotsc,\bs y_{n}'$ be the corresponding relabelling of $\bs y_1,\dotsc,\bs y_n$.

\begin{lemma}
\label{lemma:project_volumes}
\label{eq:volumes}
    \begin{align}
        \vol\bra{[\bs 0,\bs w]+\sum_{i=1}^{n-2}[\bs 0,\bs u'_i]} &= \abs{\det(\bs w,\bs u'_1,\dotsc,\bs u'_{n-2})} 
        = \abs{\det(\bs y'_1,\dotsc,\bs y'_{n-2})}.
    \end{align}
\end{lemma}

\begin{proof}
By a unimodular transformation  sending the primitive vector $\bs w$ to the last coordinate vector $\bs e_{n-1}$, there is no loss of generality in assuming the $\bs w = \bs e_{n-1}$. In this case the equality of the determinants is clear.
\end{proof}

Since $Z$ is a zonotope with $n$ generators $\bs y_1,\dotsc,\bs y_{n}$ and dimension $n-2$, 
there are the following $n^2$ natural choices of zonotopes $Z'$ in Theorem~\ref{thm:LRZinside}.
First, we consider the~$n$ zonotopes generated by all but one of the generators of~$Z$.
That is, taking~$Z'$ to be
\begin{equation}\label{eq:TypeI}
\bs y'_n + \sum_{i=1}^{n-1}[\bs0,\bs y'_i].
\end{equation}
The other $n(n-1)$ options for~$Z'$ correspond to combining two of the generators of~$Z$ into a single one, and taking the rest of the generators individually, that is, taking~$Z'$ to be of the form
\begin{equation}\label{eq:typeIIa}
 \sum_{i=1}^{n-2}[\bs0,\bs y'_i]+[\bs0,\bs y'_{n-1}+\bs y'_n]
\end{equation}
or
\begin{equation}\label{eq:typeIIb}
 \bs y'_n+\sum_{i=1}^{n-2}[\bs0,\bs y'_i]+[\bs0,\bs y'_{n-1}-\bs y'_n].
\end{equation}
Each of these $n^2$ zonotopes has $n-1$ generators and is contained in $Z$. The only missing property in order for one of them to be an LRZ (respectively, an sLRZ, which we need in Section~\ref{sec:sLRC}) is that the corresponding volume vector contains no zero entry (respectively, contains no zero entry and no two equal or opposite entries).

Some zonotope of type~\eqref{eq:TypeI} is an LRZ if and only if the volumes of the parallelepipeds spanned by~$\bs w$ and some $n-2$ vectors among the $\bs u'_i$ are nonzero; if, in addition, such volumes are all distinct, then
it is even an sLRZ.
From now on, we will express $\bs w \in \ZZ^{n-1}$ in terms of the $\bs u_i$, namely,
\begin{equation}\label{eq:wcoords}
 \bs w=\sum_{i=1}^n\rho_i\bs u_i,
\end{equation}
where $\rho_i\in\QQ$, $1\leq i\leq n$. 
Observe that the vector $(\rho_1,\dots,\rho_n)$ is not unique, but any choice will work in what follows.
Below, the numbers $\rho_i'$ and $v_j'$ correspond to the induced relabelling of the coefficients~$\rho_i$ in~\eqref{eq:wcoords} and the velocities~$v_j$ in~\eqref{eq:lindepui}.

\begin{prop}\label{prop:TypeI}
Suppose $P=\sum_{i=1}^{n-1}[\bs0,\bs u'_i]$, that is, $P$ is the parallelepiped whose projection is~\eqref{eq:TypeI} (translated by a lattice vector).
Then, the volume of the parallelepiped spanned by $\bs w$ and all the vectors~$\bs u'_i$ except from~$\bs u'_j$ equals the absolute value of
\begin{equation*}
{\begin{vmatrix}
 \rho'_j & \rho'_n\\
 v'_j & v'_n
\end{vmatrix}}, \;\;\;
1\leq j\leq n-1.
\end{equation*}
\end{prop}

\begin{proof}
 Using \eqref{eq:lindepui} and~\eqref{eq:wcoords}, the volume spanned by $\bs w$ and all $\bs u'_i$ except from $\bs u'_j$, is
 \begin{align*}
  &\abs{\det(\bs w,\bs u'_1,\dotsc,\bs u'_{j-1},\bs u'_{j+1},\dotsc,\bs u'_{n-1})} \\
  &\hspace{1cm}= \abs{\det(\bs w-\tfrac{\rho'_n}{v'_n}\sum_{i=1}^nv'_i\bs u'_i,\bs u'_1,\dotsc,\bs u'_{j-1},\bs u'_{j+1},\dotsc,\bs u'_{n-1})}\\
  &\hspace{1cm}= \abs{\det((\rho'_j-\tfrac{\rho'_n}{v'_n}v'_j)\bs u'_j,\bs u'_1,\dotsc,\bs u'_{j-1},\bs u'_{j+1},\dotsc,\bs u'_{n-1})}\\
  &\hspace{1cm}= \abs{\rho'_jv'_n-\rho'_nv'_j},
 \end{align*}
as claimed.
\end{proof}

\begin{prop}\label{prop:TypeII}
Suppose $P=\sum_{i=1}^{n-2}[\bs0,\bs u'_i]+[\bs0,\bs u'_{n-1}\pm\bs u'_n]$, that is, $P$ is the parallelepiped whose projection is~\eqref{eq:typeIIa} or~\eqref{eq:typeIIb} (translated by a lattice vector). 
Then, the volume of the parallelepiped spanned by $\bs w$ and $n-2$ vectors among the generators of $P$ equals the absolute value of
\begin{equation*}
{\begin{vmatrix}
 \rho'_j & \rho'_{n-1}\mp\rho'_n\\
 v'_j & v'_{n-1}\mp v'_n
\end{vmatrix}}, \;\;\;
1\leq j\leq n-1.
\end{equation*}
\end{prop}

\begin{proof}
 Using \eqref{eq:lindepui} and~\eqref{eq:wcoords}, the volume spanned by $\bs w$ and all generators except from $\bs u'_j$, $1\leq j\leq n-2$, is given by
 \begin{align*}
 & \abs{\det(\bs w,\bs u'_1,\dotsc,\bs u'_{j-1},\bs u'_{j+1},\dotsc,\bs u'_{n-2},\bs u'_{n-1}\pm\bs u'_n)}\\ 
 &= \abs{\det(\bs w-\tfrac{\rho'_j}{v'_j}\sum_{i=1}^nv'_i\bs u'_i,\bs u'_1,\dotsc,\bs u'_{j-1},\bs u'_{j+1},\dotsc,\bs u'_{n-2},\bs u'_{n-1}\pm\bs u'_n)}\\
  &= \abs{\det((\rho'_{n-1}-\tfrac{\rho'_j}{v'_j}v'_{n-1})\bs u'_{n-1}+(\rho'_n-\tfrac{\rho'_j}{v'_j}v'_n)\bs u'_n,\bs u'_1,\dotsc,\bs u'_{j-1},\bs u'_{j+1},\dotsc,\bs u'_{n-2},\bs u'_{n-1}\pm\bs u'_n)}\\
  &= \abs{\begin{vmatrix}
               \rho'_{n-1}-\tfrac{\rho'_j}{v'_j}v'_{n-1} & 1\\
               \rho'_n-\tfrac{\rho'_j}{v'_j}v'_n & \pm 1\\
              \end{vmatrix}
   \det(\bs u'_1,\dotsc,\bs u'_{j-1},\bs u'_{j+1},\dotsc,\bs u'_n)           
}\\
 &= \abs{\begin{vmatrix}
               \rho'_{n-1}v'_j-{\rho'_j}{}v'_{n-1} & 1\\
               \rho'_nv'_j-{\rho'_j}{}v'_n & \pm 1\\
              \end{vmatrix}
}\\
&= \abs{\begin{vmatrix}
         \rho'_j & \rho'_{n-1}\mp\rho'_n\\
         v'_j & v'_{n-1}\mp v'_n
        \end{vmatrix}
}.
\end{align*}
The volume spanned by $\bs w$ and $\bs u'_1,\dotsc,\bs u'_{n-2}$ is, by Proposition~\ref{prop:TypeI},
 \begin{align*}
  \abs{\det(\bs w,\bs u'_1,\dotsc,\bs u'_{n-2})} 
  &= \abs{\begin{vmatrix}
 \rho'_{n-1} & \rho'_n\\
 v'_{n-1} &  v'_n
\end{vmatrix}}
  = \abs{\begin{vmatrix}
 \rho'_{n-1} & \rho'_{n-1}\mp\rho'_n\\
 v'_{n-1} & v'_{n-1}\mp v'_n
\end{vmatrix}}.
\qedhere
 \end{align*}
\end{proof}

\begin{rem}
Calling $\bs p_i =(\rho_i, v_i)$ and $\bs e_1=(1,0)$ we have that the configurations 
\[
\{\bs u_1, \dots, \bs u_n, \bs w\} \subseteq \RR^{n-1}
\qquad{ \text{and} }\qquad
\{\bs p_1, \dots, \bs p_n, -\bs e_1\} \subseteq \RR^{2}
\]
are Gale dual to one another. Indeed, as explained in Remark~\ref{rem:gale}, this amounts to saying that the row-spaces of
$
\begin{pmatrix}
 \bs u_1 &\dots & \bs u_n & \bs w
\end{pmatrix}
$
and
$
\begin{pmatrix}
  \rho_1 &\dots &  \rho_n & - 1\\
  v_1 &\dots &  v_n & 0\\
\end{pmatrix}
$
are orthogonal complements of one another. That they are orthogonal follows from Eqs.~\eqref{eq:lindepui} and~\eqref{eq:wcoords}, and complementarity from adding up the ranks of the matrices, $n-1$ and~$2$ respectively. 

Now, Gale duality implies that each maximal minor in one of the matrices equals  the complementary minor in the other (modulo a multiplicative constant, the same for all minors, and a sign depending on the indices of columns used in each minor). This provides an alternative proof of Proposition~\ref{prop:TypeI}: the minor obtained in the first matrix forgetting $\bs u_i$ and $\bs u_j$ equals (modulo sign; the constant happens to be $\pm1$ for our particular ``choice of Gale transform'') the determinant of $\bs p_i$ and $\bs p_j$.
\end{rem}

\begin{proof}[Proof of Theorem~\ref{thm:LRZinside}]
It suffices to show that a zonotope of type~\eqref{eq:typeIIa} or~\eqref{eq:typeIIb} is a lonely runner zonotope. 
The center $\bs z$ of such a zonotope, obviously satisfies the desired condition. Consider now the vectors $\bs p_i=(\rho_i,v_i)\in\QQ\times\ZZ$, $1\leq i\leq n$.
 By Proposition~\ref{prop:TypeII} and Eq.~\eqref{eq:volumes}, it suffices to show that there are two of them, say~$\bs p_i$ and~$\bs p_j$, such that $\bs p_i\pm\bs p_j$, with an appropriate choice of sign,
 is not parallel to any vector $\bs p_k$, $1\leq k\leq n$.
 
 Without loss of generality, assume that 
 \[
 \frac{\rho_{n-1}}{v_{n-1}}=\max\set{\frac{\rho_i}{v_i}:1\leq i\leq n}
 \quad\textrm{ and }\quad
 \frac{\rho_n}{v_n}=\min\set{\frac{\rho_i}{v_i}:1\leq i\leq n}.
 \]
 We note that not all $\bs p_k$ are parallel, otherwise $\bs w$ would be the zero vector.
 Hence, the inequality $\frac{\rho_{n-1}}{v_{n-1}} > \frac{\rho_n}{v_n}$ is indeed strict.

 If $v_{n-1}=v_n$, then $\bs p_{n-1}-\bs p_n=(\rho_{n-1}-\rho_n,0)$ is obviously not parallel to any of the~$\bs p_i$, as their $y$-coordinates are nonzero.
 If $v_{n-1}>v_n$, then $\bs p_{n-1}-\bs p_n$ is on the upper half plane, as all~$\bs p_i$ are, and it holds
 \[
 \frac{\rho_{n-1}-\rho_n}{v_{n-1}-v_n}>\frac{\rho_{n-1}}{v_{n-1}},
 \]
 whence by definition of~$\bs p_{n-1}$ it follows that $\bs p_{n-1}-\bs p_n$ is not parallel to any $\bs p_k$, $1\leq k\leq n$.
 A similar argument with $\bs p_n - \bs p_{n-1}$ applies if $v_{n-1}<v_n$.
\end{proof}


\subsection{Proof of Theorem~\ref{thm:main-points}}
\label{finitenessLRC-points}
To have a more symmetric description, from a given LRZ $Z_{\bs v} \subseteq \RR^{n-1}$ with center $\bs x_{\bs v}$, we define the $\bs 0$-symmetric scaled down zonotope
\[
K_{\bs v}=\frac{n-1}{n+1}(Z_{\bs v}-\bs x_{\bs v}).
\]
Conjecture~\ref{lrcgeom} is then equivalent to the statement that
\begin{equation}
K_{\bs v}\cap(\bs x_{\bs v}+\ZZ^{n-1})\neq\vn.
\end{equation}
Observe that, by Lemma~\ref{lemma:minkowski-first-points} plus the hypothesis that $Z_{\bs v}$ has more than $\binom{n+1}{2}^{n-1}$ lattice points, we have that
\begin{align}
\lambda_1(K_{\bs v}) = 2 \lambda_1(K_{\bs v}-K_{\bs v}) = 
\frac{2(n+1)}{n-1} \lambda_1(Z_{\bs v}-Z_{\bs v}) \le 
\frac{4}{(n-1)n}.
\label{eq:lambda-K}
\end{align}
Now, let $\bs w\in\ZZ^{n-1}$ be a vector attaining the first successive minimum of $K_{\bs v}$, that is,
\[
\bs w\in \la_1(K_{\bs v})K_{\bs v}\cap(\ZZ^{n-1}\sm\set{\bs 0}).
\]
The vector~$\bs w$ is necessarily primitive; hence,
if we consider a projection $T:\RR^{n-1} \to \RR^{n-2}$ with $T(\ZZ^{n-1}) = \ZZ^{n-2}$ and with $T(\bs w)=\bs 0$, 
Theorem~\ref{thm:LRZinside} tells us that the image  $Z=T(Z_{\bs v})$ contains an LRZ with the same center as $Z$. 

We keep the notation $\bs y_i:=T(\bs u_i)$, so that 
\[
Z = \sum_{i=1}^n\ent{\bs 0,\bs y_i}
\]
and the center of $Z$ is given by $\bs y_{\bs v}:=T(\bs x_{\bs v}) = \frac12\left(\bs y_1 + \dotsc + \bs y_n\right)$.
Then, the assumption that Conjecture~\ref{lrcgeom} holds in dimension $n-2$ and that $Z$ contains an LRZ imply that 
\[
\frac{n-2}{n}(Z-\bs y_{\bs v})\cap(\bs y_{\bs v}+\ZZ^{n-2})\neq \vn.
\]

Let $\bs b$ be a vector in the above intersection, and let $\bs a\in \frac{n-2}{n}(Z_{\bs v}-\bs x_{\bs v})$ be such that $T(\bs a)=\bs b$.
Since $\bs b - \bs y_{\bs v} \in \ZZ^{n-2} =T(\ZZ^{n-1})$, we have that $\bs a - \bs x_{\bs v} \in T^{-1}(\ZZ^{n-2})$. That, is, the line $\bs a + \RR \bs w=T^{-1}(\bs b)$  contains infinitely many points of $\bs x_{\bs v}+\ZZ^{n-1}$, forming an affine one-dimensional lattice.
Every segment between two consecutive lattice points on this line has lattice length equal to~$1$.
Hence, in order to ensure that the intersection $K_{\bs v}\cap(\bs x_{\bs v}+\ZZ^{n-1})$ is nonempty, it suffices to show:

\begin{lemma}
The lattice length of the segment $K_{\bs v}\cap(\bs a+\RR\bs w)$ is at least $1$.
\end{lemma}

\begin{proof}
We abbreviate $\lambda_1(K_{\bs v})$ as $\la_1$.
If $\bs a\in\RR\bs w$, then $K_{\bs v}\cap(\bs a+\RR\bs w) = K_{\bs v}\cap \RR\bs w$, which has lattice length $2/\la_1$. Equation~\eqref{eq:lambda-K} shows that this is at least $n(n-1)/2 \ge 1$, since $n\ge 2$.

If, on the contrary, $\bs a\notin\RR\bs w$, then~$\bs w$ and~$\bs a$ generate a two-dimensional linear subspace.
Consider the point
\[
\bs a':=\frac{(n-1)n}{(n-2)(n+1)}\bs a\in\frac{n-1}{n+1}(Z_{\bs v}-\bs x_{\bs v})=K_{\bs v}.
\]
The triangle~$S$  with vertices $\bs a'$, $\frac{1}{\la_1}\bs w$, and $-\frac{1}{\la_1}\bs w$ is contained in $K_{\bs v}$, so we have
\[\ell(K_{\bs v}\cap(\bs a+\RR\bs w))\geq \ell(S\cap(\bs a+\RR\bs w)).\]
Now, 
\[
\bra{\frac{(n-2)(n+1)}{(n-1)n}} +  \bra{\frac{2}{(n-1)n}}= 1,
\]
and
\[
\bra{\frac{(n-2)(n+1)}{(n-1)n}} \bs a'  \pm  \bra{\frac{2}{(n-1)n}} \frac{1}{\la_1}\bs w= \bs a \pm \frac{2}{(n-1)n\la_1}\bs w,
\]
imply that this  intersection $S\cap(\bs a+\RR\bs w)$ is precisely the segment
\[\ent{\bs a-\frac{2}{(n-1)n\la_1}\bs w,\bs a+\frac{2}{(n-1)n\la_1}\bs w},\]
whose lattice length is precisely
\[\frac{4}{(n-1)n\la_1}.\]
This is at least $1$  by Eq.~\eqref{eq:lambda-K}.
\end{proof}


\section{The Lonely Vector Problem,  and a finite checking result for sLRC}
\label{sec:LVP}
\label{sec:sLRC}

\subsection{Proof of Theorem~\ref{thm:mainslrc-points}}
\label{subsec:finiteness-sLRC-points}

In this section, we prove Theorem~\ref{thm:mainslrc-points} (and thus Theorem~\ref{thm:mainslrc}).
We keep the notation from the previous section and, in particular, we consider the set of vectors
\[
\bs P=\set{\bs p_i=(\rho_i,v_i):1\le i\le n} \subseteq \QQ\times\ZZ
\]
from the proof of Theorem~\ref{thm:LRZinside}.
They are not all parallel to each other since that would imply the vectors $(v_1,\dots,v_n)$ and $(\rho_1,\dotsc,\rho_n)$ to be parallel, 
which cannot happen because
\[
\sum_{i=1}^n v_i \bs u_i = \bs 0,
\qquad
\sum_{i=1}^n\rho_i \bs u_i = \bs w \ne \bs 0.
\]
Since we deal with Conjecture~\ref{slrc}, or its equivalent geometric formulation, Conjecture~\ref{slrcgeom}, we may further assume that the~$v_i$ are positive and pairwise distinct. In particular, no two vectors from $\bs P$ are equal or opposite.

Our goal is to prove that  the zonotope $Z_{\bs v}$ contains an sLRZ of type~\eqref{eq:TypeI}, \eqref{eq:typeIIa} or~\eqref{eq:typeIIb}, but we are only able to do this assuming that $\bs P$ has the Lonely Vector Property (LVP) introduced in Definition~\ref{defn:LVP}. 

\goodbreak
\begin{prop}
\label{prop:sLRZinside}
 Suppose that 
 \begin{enumerate}
  \item $Z=\sum_{i=1}^n[\bs0,\bs u_i]\subseteq\RR^{n-1}$ is an sLRZ, and that
  \item the set of vectors $\bs P=\set{\bs p_i=(\rho_i,v_i):1\le i\le n}$ satisfies the Lonely Vector Property, where~$(v_1,\dots,v_n)$ is the volume 
 vector of $Z$  and $\bs w=\sum_{i=1}^n\rho_i\bs u_i\in \ZZ^{n-1}$ is a lattice vector attaining the first successive minimum of~$Z-Z$.
 \end{enumerate}
Then, $T(Z)$ contains an sLRZ of type~\eqref{eq:TypeI}, \eqref{eq:typeIIa} or~\eqref{eq:typeIIb}.
\end{prop}

\begin{proof}
 We distinguish two cases: assume first that one of the vectors, without loss of generality $\bs p_n$, is not parallel to any nonzero vector of the form $\bs p_k\pm\bs p_\ell$, where $(k,\ell)\neq(n,n)$. Then, the zonotope
 \[Z'=\sum_{i=1}^{n-1}[\bs 0,\bs y_i],\]
 where $T(\bs u_i)=\bs y_i$ is an sLRZ.
 Indeed, if two volumes of parallelepipeds of $Z'$ were equal, then by Lemma~\ref{lemma:project_volumes} and Proposition~\ref{prop:TypeI} we would have
 \[\begin{vmatrix}
    \rho_i & \rho_n\\
    v_i & v_n
   \end{vmatrix}=\pm\begin{vmatrix}
    \rho_j & \rho_n\\
    v_j & v_n
   \end{vmatrix},
\]
for some $1\leq i<j\leq n-1$, or equivalently, $\bs p_n$ would be parallel to $\bs p_i\pm\bs p_j$, contradicting the assumption on $\bs p_n$. If one volume were zero, then $\bs p_n$ would be parallel to some $\bs p_i$,
$1\leq i\leq n-1$, again a contradiction.

For the second case, we assume that $\bs p_{n-1}\pm\bs p_n$ is not parallel to any nonzero vector of the form $\bs p_k\pm\bs p_\ell$, where $(k,\ell)\neq (n-1,n)$. Then, the zonotope
 \[Z'=\sum_{i=1}^{n-2}[\bs 0,\bs y_i]+[\bs0,\bs y_{n-1}\mp\bs y_n],\]
 is an sLRZ.
 Indeed, if two volumes of parallelepipeds of $Z'$ were equal, then by Lemma~\ref{lemma:project_volumes} and Proposition~\ref{prop:TypeII} we would have
 \[
 \begin{vmatrix}
    \rho_i & \rho_{n-1}\pm\rho_n\\
    v_i & v_{n-1}\pm v_n
   \end{vmatrix}=\pm\begin{vmatrix}
    \rho_j & \rho_{n-1}\pm\rho_n\\
    v_j & v_{n-1}\pm v_n
   \end{vmatrix},
\]
for some $1\leq i<j\leq n-1$, or equivalently, $\bs p_{n-1}\pm\bs p_n$ would be parallel to $\bs p_i\pm\bs p_j$, contradicting the assumption on $\bs p_n$. If one volume were zero, then $\bs p_{n-1}\pm\bs p_n$ 
would be parallel to some $\bs p_i$, $1\leq i\leq n-1$, again a contradiction, completing the proof.
\end{proof}

\begin{proof}[Proof of Theorem~\ref{thm:mainslrc-points}]
We suppose that Conjecture~\ref{slrcgeom}  holds for  $n-1$, so in particular it holds for the sLRZ $Z' \subseteq T(Z)$, which exists by Proposition~\ref{prop:sLRZinside}.
On the other hand, by Lemma~\ref{lemma:minkowski-first-points}, $\la_1(Z-Z) \le \frac{2}{n(n+1)}$. Thus, Proposition~\ref{prop:CSS-sum-bound} gives
\begin{align*}
\mu(Z) &\leq \la_1(Z-Z)+\mu(T(Z),\ZZ^{n-2}) \\ 
&\le \ \la_1(Z-Z)+\mu(Z',\ZZ^{n-2})\\
&\leq \frac{2}{n(n+1)}+\frac{n-2}{n} \ = \ \frac{n-1}{n+1}.\qedhere
\end{align*}
\end{proof}

\subsection{Small cases of the LVP}
\label{subsec:LVP-small}

\noindent We now  tackle a couple of special cases for which the LVP holds.
To this end, recall from the introduction that for a point set $\bs P=\set{\bs p_1,\dotsc,\bs p_n}\subseteq \RR^2$ we associate the multiset
\[
S_{\bs P} = \bs P \cup \set{\bs p_i+\bs p_j : 1\leq i<j\leq n} \cup \set{\bs p_i-\bs p_j : 1\leq i<j\leq n}.
\]

\noindent In the first case, we do not require the vectors to be rational.

\begin{prop}\label{prop:allbuttwocollinear}
Suppose that $\bs P=\set{\bs p_1,\dotsc,\bs p_n}$ linearly spans $\RR^2$, contains no two equal or opposite elements and  $\bs p_3,\dotsc,\bs p_n$ are parallel. Then, $\bs P$ has the LVP.
\end{prop}

\begin{proof}
If all but one vector were parallel, say $\bs p_2$ is also parallel to $\bs p_3,\dotsc,\bs p_n$, then all vectors of the form $\bs p_j\pm\bs p_k$ are parallel to $\bs p_2$ as well,
for $2\leq j,k\leq n$. Then, obviously each vector $\bs p_1\pm \bs p_j$ is not parallel to any other vector in $S_{\bs P}$, for $2\leq j\leq n$.

If $\bs p_1$ and $\bs p_2$ were parallel, then we write every vector in $S_{\bs P}$ as a linear combination of $\bs p_1$ and $\bs p_3$. We may further assume that $\bs p_2=\la_2\bs p_1$,
with $\la_2>\la_1=1$, and $\bs p_j=\mu_j\bs p_3$ for $4\leq j\leq n$, with $1=\mu_3<\mu_4<\mu_5<\dotsb<\mu_n$. Then, it is clear that $\bs p_2+\bs p_3=\la_2\bs p_1+\bs p_3$ 
is not parallel to any other vector of $S_{\bs P}$; indeed, if there were such a vector in $S_{\bs P}$, it should have both coordinates positive, with respect to $\bs p_1$ and~$\bs p_3$, so,
it would be of the form $\bs p_i+\bs p_j=\la_i\bs p_1+\mu_j\bs p_3$, with $1\leq i\leq2$, $3\leq j\leq n$. However, if $(i,j)\neq(2,3)$, then these vectors cannot be parallel, as $\la_i-\mu_j\la_2<0$.

So, we reduce to the case where $\bs p_1$, $\bs p_2$ and $\bs p_3$ are pairwise not parallel.
 We first note that changing a $\bs p_i$ to its opposite or applying a linear transformation to $\bs P$ does not affect the LVP.
So, without loss of generality, we may assume that all $\bs p_j$ with $3\leq j\leq n$ lie on the positive $y$-axis, and $\bs p_1$, $\bs p_2$ in the first quadrant, such that~$\bs p_1$ has smaller slope than~$\bs p_2$;
 We also assume that
 \[
 v_3<v_4<\dotsb<v_n.
 \]
With this convention, the slopes of the vectors $\bs p_2+\bs p_j$, $3\leq j\leq n$, form a strictly increasing sequence of numbers strictly between the slope of~$\bs p_2$ and the slope of~$\bs p_3$, while those of $\bs p_1-\bs p_j$, $3\leq j\leq n$, form a strictly decreasing sequence of numbers strictly between the slope of~$\bs p_1$ and the slope of~$-\bs p_3$. Therefore, these $2(n-2)$ vectors along with $\bs p_1$, $\bs p_1+\bs p_2$, $\bs p_2$, define $2n-1$ distinct lines through the origin, none of them the $y$-axis.
There are a total of $(n-2)^2$ vectors of $S_{\bs P}$ lying on the $y$-axis, so failure of the LVP and the pigeonhole principle would imply them to lie in at most $2n-2$ lines, 
completing the proof that~$\bs P$ indeed satisfies the LVP.
\end{proof}

\begin{cor}
 Any set of three vectors spanning $\RR^2$, with no two equal or opposite, has the LVP.
\end{cor}

We now deal with the case of four vectors, for which we need them to be rational.

\begin{prop}
Let $\bs P = \set{\bs p_1,\bs p_2,\bs p_3,\bs p_4}\subseteq \QQ^2$, with not all parallel and no two equal or opposite. Then, $\bs P$ has the LVP.
\end{prop}

\begin{proof}
 If two of the vectors are parallel then $\bs P$ has the LVP, by Proposition~\ref{prop:allbuttwocollinear}. So, we may assume that no two of them are parallel.

Since the LVP is invariant under linear transformation and under changing one or more vectors to their opposites, we can also assume that 
\[
\bs p_1=(1,0), \ \bs p_2=(x_2,y_2), \ \bs p_3=(x_3,y_3), \ \bs p_4=(0,1),
\]
with $x_2, x_3, y_2, y_3 >0$ and the slope of $\bs p_2$ smaller than that of $\bs p_3$. Indeed, first perform a rational linear transformation sending 
$\bs p_4$ to $(0,1)$; then assume without loss of generality that $\bs p_1, \bs p_2, \bs p_3$ have positive $x$-coordinate (changing them to their opposite if needed) and are ordered according to slope (relabelling them if needed); finally, send $\bs p_1$ to $(1,0)$ with a second rational linear transformation that fixes $\bs p_4$.

The vectors $\bs p_i$ ($i=1,2,3,4$),  $\bs p_i + \bs p_{i+1}$ ($i=1,2,3$) and $\bs p_4 - \bs p_{1}$ define eight distinct lines through the origin; the first seven have non-negative distinct slopes (including $0$ and $\infty$) since
 \[0 \le \frac{x_i}{y_i}<\frac{x_i+x_{i+1}}{y_i+y_{i+1}}<\frac{x_{i+1}}{y_{i+1}} \le \infty, \;\;\; 1\leq i\leq3,\]
 and the last one has negative slope, equal to $-1$.
 
Since $S_{\bs P}$ has $16$ elements, the only possibility for $\bs P$ not to satisfy the LVP would be if each of these eight lines contains \emph{exactly} two of the vectors; below we show that this possibility leads to a contradiction.

If one of $x_i$ ($i\in \{2,3\}$) is smaller than $1$ then $\bs p_i - \bs p_1$ has negative slope, hence it must have the slope of $\bs p_4 - \bs p_1$, which gives a contradiction since then the three vectors $\bs p_4 - \bs p_i$, $\bs p_4 - \bs p_1$ and $\bs p_i - \bs p_1$ are parallel.
So, $x_2,x_3\ge 1$. The same argument with $\bs p_4 - \bs p_i$ instead of $\bs p_i - \bs p_1$ implies $y_2,y_3\ge 1$.
This now implies that 
\begin{itemize}
\item The only vector of $S_{\bs P}$ that can have negative slope (in particular, the only one that can be parallel to $\bs p_4-\bs p_1$) is 
$\bs p_3-\bs p_2$; hence we assume $x_2+x_3 = y_2+y_3$.
\item The only vector that can be parallel to $\bs p_1$ is $\bs p_4-\bs p_2$; hence $y_2=1$.
\item  The only that can be parallel to $\bs p_4$ is $\bs p_3-\bs p_1$; hence $x_3=1$.
\end{itemize}
 So, we we have that
\begin{align}
\label{eq:octagon}
\bs p_2=(\la,1), \;\;\; \bs p_3=(1,\la),
\end{align}
for some \emph{rational} $\la >1$ ($\la=1$ would give $\bs p_2=\bs p_3)$.

So far we have four pairs of parallel vectors: the three pairs mentioned above (with slopes $0$, $\infty$ and $-1$), plus  $\bs p_1+\bs p_4=(1,1)$ and $\bs p_2+\bs p_3=(1+\la, 1+\la)$ (with slope $1$).
The eight unpaired vectors are the following:
\[
\begin{array}{llll}
\bs p_1 + \bs p_2, &  \bs p_2,   & \bs p_3, & \bs p_3 + \bs p_4, \\
\bs p_1 + \bs p_3, & \bs p_2 + \bs p_4, &
\bs p_2 - \bs p_1, & \bs p_4 - \bs p_3.
\end{array}
\]
The first four have different (and increasing) slopes, so we would need to pair the last four to them.

The vector $\bs p_1+\bs p_3=(2,\la)$ 
has slope strictly between 
those of $\bs p_1$ and $\bs p_3$. Hence, it must be parallel either to $\bs p_2 =(\la,1)$ or to $\bs p_1+\bs p_2=(\la+1,1)$. The first case yields $\la=\sqrt{2}$, a contradiction, since $\la \in\QQ$. The second case yields $\la^2+\la-2=0$, a contradiction, since the two solutions are $\lambda\in\{1,-2\}$ and we had
$\lambda >1$. Thus, we establish a contradiction if we assume that $\bs P$ does not have the LVP, concluding the proof.
\end{proof}

\begin{rem}
As we said in the introduction, if we remove the restriction that the given vectors are rational then the vertices of any regular $n$-gon other than the triangle, hexagon, and square provide a counter-example to the LVP. In fact, our proof of the LVP for four vectors shows that the only sets of four vectors, no two parallel or opposite, failing to have the LVP are those of the form expressed in~\eqref{eq:octagon} with $\lambda = \sqrt{2}$, which are (linearly isomorphic) to four consecutive vertices of 
the regular octagon.
\end{rem}


\section[Cosimple configurations and a generalization of sLRC]{Cosimple configurations and a generalization of sLRC, with a finite checking result}
\label{sec:cosimple}

Recall Definition~\ref{defn:cosimple}: A lattice zonotope spanning $\RR^d$ is called \emphd{cosimple} 
if there is a linear dependence among its generators  having coefficients that are all non-zero and with pairwise different absolute values.
In particular, a cosimple zonotope in $\RR^d$ with $d+1$
generators has its generators in linear general position, so that the
class of cosimple zonotopes with $d+1$ generators coincides with the
class of sLRZs.

\subsection{Proof of Theorem~\ref{thm:finite-checking-cosimple-conditional}}
\label{finiteness-cosimple}

The advantage of allowing for more than $d+1$ generators is that now any projection  of a cosimple zonotope is cosimple.
Observe that the projection along one of the generators will make that generator be zero.
We can still consider it part of the vector set although this is irrelevant both for the definition of cosimplicity (we can give the zero vector an arbitrary coefficient in a linear dependence) and for Conjecture~\ref{conj:cosimple} (the zero vector as a generator does not affect the covering radius).

Invariance under projection makes the analogue of Theorems~\ref{thm:main-points} and ~\ref{thm:mainslrc-points} be much easier to prove and allows us to remove the ``Lonely Vector Problem'' from the latter:

\begin{proof}[Proof of Theorem~\ref{thm:finite-checking-cosimple-conditional}]
Let $Z$ be our cosimple zonotope with more than ${\binom{d+2}2}^d$ lattice points.
By Lemma~\ref{lemma:minkowski-first-points} we have $\lambda_1(Z-Z) \le1/ \binom{d+2}2$ and by
Proposition~\ref{prop:CSS-sum-bound} and the induction hypothesis (using that the projection of a cosimple zonotope is cosimple)
\[
\mu(Z) \leq \lambda_1(Z-Z) + \mu\left(\pi(Z),\pi(\ZZ^d)\right) \le
\frac{1}{{\binom{d+2}2}} + \frac{d-1}{d+1} =\frac{d}{d+2}.\qedhere
\]
\end{proof}

\subsection{Remarks on cosimple zonotopes}
\label{remarks-cosimple}

We first show that in Conjecture~\ref{conj:cosimple}  there is no loss of generality in assuming the generators to be primitive, generalizing Proposition~\ref{prop:allbut1gcd}:

\begin{lemma}
\label{lemma:primitive}
Let $d\ge 2$.
If Conjecture~\ref{conj:cosimple} holds in dimension $d-1$ for every cosimple zonotope with $n-1$ generators, then it holds in dimension $d$ for all cosimple zonotopes with~$n$ generators one of which is not primitive or two of which agree or are opposite to one another.
\end{lemma}

\begin{proof}
If two generators agree or are opposite to one another, then we can substitute their sum for the two of them, so we are in the case of a non-primitive generator.

Hence, let $Z$ be a cosimple $d$-zonotope with $n$ generators $\bs u_1, \dotsc, \bs u_n \in \ZZ^d$ and suppose without loss of generality that $\bs u_{n}=k \bs e_d$, for some $k\in \ZZ_{\ge 2}$.
Let $\pi:\RR^d\to \RR^{d-1}$ be the projection that forgets the last coordinate.
Then $Z=Z_1+ Z_2$, where~$Z_1$ is the segment of length~$k$ in the last coordinate direction and~$\pi(Z_2)$ is a cosimple $(d-1)$-zonotope with $n-1$ generators.
Hence, Proposition~\ref{prop:max-bound} plus the inductive hypothesis gives
\[
\mu(Z) \le \max\left\{\mu(Z_1), \mu(\pi(Z_2)) \right\} \le \max\left\{\frac1k, \frac{d-1}{d+1}\right\} \le \frac{d}{d+2},
\]
since $d \geq 2$.
\end{proof}

\begin{cor}
\label{cor:primitive}
Let $d\ge 2$.
If Conjecture~\ref{conj:cosimple} holds in dimension $d-1$ for every cosimple zonotope with $n-1$ generators, then it holds in dimension $d$ for all cosimple zonotopes with~$n$ generators one of which is not primitive or two of which agree or are opposite to one another.
\end{cor}

We now characterize cosimple configurations:

\begin{lemma}
\label{lemma:cosimple}
Let $\bs A$ be a finite collection of vectors spanning $\RR^d$. Then, $\bs A$ is cosimple if and only if neither of the following conditions hold:
\begin{enumerate}[(i)]
\item There is a hyperplane $H$ containing all but one of the elements of $\bs A$.
\item There is a hyperplane $H$ containing all but two of the elements of $\bs A$, and the two elements are at the same distance from $H$.
\end{enumerate}
\end{lemma}

\begin{proof}
The ``only if'' is easy: if (i) happens then the element in question has coefficient zero in every linear dependence.
If (ii) happens then the coefficients of the two elements in question have the same absolute value in every linear dependence.

For the converse, suppose that none of (i) and (ii) happens. Let $L \subseteq \RR^{\bs A}$ be the linear space of linear dependence vectors in $\bs A$. The fact that (i) and (ii) do not hold means that $L$ is not contained in any of the hyperplanes where a coordinate is zero or where two coordinates have the same absolute value (this arrangement of hyperplanes happens to be the Coxeter arrangement of type $B_n$, although we do not need this).
Then, a generic vector $\lambda$ from $L$ does not belong to any of those hyperplanes, and certifies that~$\bs A$ is cosimple.
\end{proof}

\begin{rem}
\label{rem:gale}
Let us call a vector configuration \emphd{simple} if no element is zero or equal or opposite to another one.\footnote{This is close, but not the same, as what simple means in matroid theory. Since a matroid forgets the lengths of vectors and remembers only their (in)dependence, in matroid theory the word ``simple'' excludes also proportional vectors, not only those that are equal or opposite.}
The characterization in Lemma~\ref{lemma:cosimple} then says that~$\bs A$ is cosimple if and only if its \emphd{Gale transform} is simple, which explains the name.

To understand this connection, let us briefly review Gale duality; see, e.g., \cite[Section 4.1]{DLRS2010triangulations} for more details.
If $\bs A = \{\bs u_1, \dots, \bs u_n\} $ is a finite set of~$n$ vectors in~$\RR^d$, 
we call \emphd{linear evaluations} and \emphd{linear dependences} of $\bs A$ the following two linear subspaces of $\RR^n\cong \RR^{\bs A}$:
\begin{align*}
\eval(\bs A)&:=\left\{\left(f(\bs u_1), \dots, f(\bs u_n)\right) :  f \in ({\RR^d})^*\right\}, \\
\dep(\bs A)&:=\left\{\left(\la_1,\dots,\la_n\right) \in \RR^n: \la_1 \bs u_1 +  \dots +  \la_n \bs u_n = \bs 0 \right\}.
\end{align*}
These two subspaces are orthogonal complements of one another, and have ranks equal to $k$ and $n-k$, where $k$ is the rank of $\bs A$. (In what follows we assume $k=d$).
A subset $\bs B\subseteq \RR^{n-d}$ of size $n$ is called a \emphd{Gale transform} or a \emphd{Gale dual} of $\bs A$ if it has the following two equivalent properties:%
\footnote{Although this is not relevant in this paper, the definition of Gale duality includes an implicit bijection between the elements of $\bs A$ and~$\bs B$. That is to say, $\bs A$ and $\bs B$ are considered labelled and Gale duality takes the labelling into account. Also, the Gale dual of a \emph{set} of vectors may have repeated vectors, hence being a \emph{multiset}. These two aspects, labelling and the possibility of repeated elements, can simultaneously be taken into account regarding $\bs A$ and $\bs B$ not as sets of vectors but rather as matrices of sizes $d\times m$ and $(m-d)\times m$, whose columns are the ``vector configurations'' we are interested in. This is the point of view taken in~\cite{DLRS2010triangulations}; see, e.g., Sections 2.1 and 2.5 in that book.}
\[
\eval(\bs A) = \dep(\bs B), 
\qquad
\eval(\bs B) = \dep(\bs A). 
\]
In this language,  Lemma~\ref{lemma:cosimple} says that $\bs A$ is cosimple if and only if none of $\bs e_i$ or $\bs e_i \pm \bs e_j$ lie in $\eval(\bs A)$, where $\bs e_i$ denotes the $i$-th standard basis vector. The same conditions for $\dep(\bs B)$ are equivalent to $\bs B$ being simple. 
\end{rem}

\begin{cor}
\label{coro:d+2vectors}
Any set of integer vectors in linear general position with at least two more vectors than its dimension is cosimple.
\end{cor}

\begin{proof}
Assuming, for contradiction, that the vectors are not cosimple, the obstructions in Lemma~\ref{lemma:cosimple} imply that all but (at most) two of them lie in a hyperplane, hence they are not in linear general position.
\end{proof}

\begin{cor}
\label{cor:widthge3}
Every cosimple zonotope has lattice-width three or more. 
Conversely, if $Z$ is a lattice $d$-zonotope of width three or more, not a parallelepiped, and its generators span the lattice, then $Z$ is cosimple.
\end{cor}

\begin{proof}
For the first part we argue by contradiction.
If $f$ is an integer linear functional giving width $w\in\{1,2\}$ to a lattice zonotope~$Z$ there are two possibilities: either $f$ is zero in all but one of the generators, or it has value $\pm 1$ in two of them and is zero in the rest. Both cases imply $Z$ not to be cosimple, by Lemma~\ref{lemma:cosimple}.

For the second part, suppose that $Z$ is not cosimple and that the generators of $Z$  integrally span $\ZZ^d$. By Lemma~\ref{lemma:cosimple} one of the following happens:
\begin{enumerate}
\item All but one of the generators lie in a hyperplane. Then the condition that the generators span the lattice implies that $Z$ has width one with respect to that hyperplane.

\item All but two generators lie in a hyperplane, and the functional $f$ vanishing on that hyperplane has the same value on the other two generators. Again, the condition that generators span the lattice implies that $f$ has value $\pm 1$ on those two generators, hence $Z$ has width two.\qedhere
\end{enumerate}
\end{proof}

\noindent Lemma~\ref{lemma:cosimple} implies that being cosimple is closed under extending the set.
Since the covering radius is non-increasing with respect to inclusion, minimal counter-examples to Conjecture~\ref{conj:cosimple} must be also \emphd{minimal cosimple configurations}; that is, cosimple configurations with the property that removing an element produces a non-cosimple one.
To this end, in the following result we characterize minimal cosimple configurations, except that the characterization is easier to express in terms of Gale duality:

\begin{cor}
Let $\bs A$ be a cosimple configuration and let $\bs B$ be its Gale transform.
Then, $\bs A$ is minimal cosimple if and only if in $\bs B$ every element is parallel to either another element or to the sum or difference of some other two elements.
\end{cor}

\begin{proof}
Via Gale duality, deleting an element $\bs u$ in a configuration $\bs A$ is equivalent to \emph{contracting} the corresponding element $\bs v$ in its Gale dual $\bs B$, which geometrically amounts to projecting $\bs B$ along the direction of $\bs v$ (see \cite[Section 4.2]{DLRS2010triangulations}). Hence, we want to characterize the simple configurations $\bs B$ with the property that the projection of $\bs B \setminus \bs v$ along the direction of $\bs v$
fails to be simple for every $\bs v \in \bs B$. 

The failure may happen in two ways: either $\bs B$ has two parallel elements, so that projecting along one of them makes the other one zero, or $\bs B$ has two elements with sum or difference  parallel to a third element, so that projecting along the latter makes the first two equal or opposite. 
\end{proof}

This characterization of minimal cosimple configurations shows that the zonotopes of type~\eqref{eq:TypeI}, \eqref{eq:typeIIa} and~\eqref{eq:typeIIb} are the natural ones to consider in Sections~\ref{sec:LRC} and~\ref{sec:sLRC}.


\section{Dimension two: proof of sLRC and its cosimple generalization}
\label{sec:dim2}

That Conjecture~\ref{slrcgeom} holds in dimension two (equivalently, that the shifted LRC holds for four runners) has been proved in~\cite{cslovjecsekmalikiosisnaszodischymura2022computing}.
Here, we give a proof of the stronger Conjecture~\ref{conj:cosimple} in dimension two and for an arbitrary number~$n$ of generators.
The idea is to show that we only need to look at zonotopes of very small volume.
For this, we use two immediate consequences of Lemma~\ref{lemma:AverkovWagner}:

\begin{cor}\ 
\label{cor:AverkovWagner}
\begin{enumerate}[(i)]
 \item Let $C \subseteq \RR^2$ be a convex body of lattice-width at least $w$ and covering radius greater than $\mu$, with $\mu w >1$. Then,
\[
\vol(C) < \frac{w^2}{2\mu w - 2}.
\]
 \item Let $Z$ be a lattice $2$-zonotope of lattice-width $w \ge 3$ and covering radius $\mu > 1/2$. Then,
\[
w = 3 \quad , \quad \mu \leq \frac23 \quad \textrm{and} \quad \vol(Z) \leq 8.
\]
\end{enumerate}
\end{cor}

\begin{proof}
(i): This is basically a rephrasing of Lemma~\ref{lemma:AverkovWagner}.
Let $\mu'>\mu$ be the covering radius of $C$ and let $C'=\mu' C$, which has a hollow translate and lattice-width at least~$\mu' w > 1$. Then, Lemma~\ref{lemma:AverkovWagner} gives
\[
\vol(C) = \frac{\vol(C')}{{\mu'}^2} \le \frac{{\mu'}^2 w^2}{{\mu'}^2 2({\mu'} w-1)} = \frac{w^2}{ 2\mu' w-2}  < \frac{w^2}{ 2\mu w-2}.
\]
(ii): As before, the zonotope $Z' = \mu Z$ has a hollow translate, and by the assumptions, its lattice-width equals $\mu w$ and is strictly greater than $3/2$.
Moreover, we also have $\mu w \leq 2$ by the two-dimensional ``flatness theorem'' for $\bs 0$-symmetric convex bodies due to~\cite[Corollary~2.7]{AverkovWagner}.
In combination, the inequalities $3/2 < \mu w \leq 2$ and $\mu > 1/2$ give $w=3$ and $\mu \leq 2/3$, because the lattice-width of any lattice zonotope is an integer.

Regarding the volume, Lemma~\ref{lemma:AverkovWagner} together with $1/\mu < 2$ gives
\[
\vol(Z) = \frac{\vol(Z')}{\mu^2} < \frac{4(\mu w)^2}{2(\mu w - 1)} \leq 9.
\]
The last inequality holds, since the function $x \mapsto \frac{x^2}{x-1}$ has maximum value equal to~$9/2$ in the interval $x \in [\frac32,2]$.
Finally, the volume of any lattice zonotope is an integer, so that $\vol(Z) \leq 8$ as claimed.
\end{proof}

For the proof of our main two-dimensional result we need the following classification of lattice parallelograms with primitive generators.
The classification is up to affine transformations that do not change the lattice~$\ZZ^d$.
More precisely, two lattice polytopes $P,Q \subseteq \RR^d$ are called \emphd{unimodularly equivalent} if there is a unimodular matrix $U \in \ZZ^{d \times d}$, that is, $|\det(U)|=1$, and a translation vector $\bs t \in \ZZ^d$ such that $P = UQ + \bs t$; in symbols $P \cong Q$.

\begin{lemma}
\label{lemma:p-q}
Let $P$ be a lattice parallelogram of area $q$ with primitive generators. Then, $P$ is unimodularly equivalent to 
\[
P_{p,q}:=[\bs 0, \bs e_1] + [\bs 0, (p,q)] ,
\]
for some $p\in \ZZ$ with $\gcd(p,q)=1$. Moreover, $P_{p,q}$ and $P_{p',q}$ are unimodularly equivalent 
if  $p'=\pm p^{\pm 1} \bmod q$.

Hence, if $q\le 8$, then $P$ is equivalent to either $P_{1,q}$ or to one of 
$P_{2,5}$, $P_{2,7}$, $P_{3,8}$. 
\end{lemma}

\begin{proof}
Since our generators are primitive, there is no loss of generality in assuming that the first one is $\bs e_1$, which already implies that our parallelogram is a lattice translation of $P_{p,q}$ for some $p$. The condition $\gcd(p,q)$ is neccessary (and sufficient) for the generator $(p,q)$ to be primitive.

The unimodular transformation $\binom{1\ 1}{0 \ 1}$ shows that $P_{p,q} \cong P_{p+q,q}$; hence  $p$ is only important modulo $q$.
That $P_{p,q} \cong P_{-p,q}$ follows by reflection on the $y$-axis (followed by the translation by the vector~$\bs e_1$) and that $P_{p,q} \cong P_{p',q}$ when $p'=p^{-1} \bmod q$ follows from the fact that if $p'p =1 +aq$ for some integer $a$, then the unimodular transformation $\binom{p'\ -a}{q\ \ -p}$ sends $\{\bs e_1, (p,q)\}$ to $\{(p',q), \bs e_1\}$.
\end{proof}

\begin{thm}
\label{thm:cosimple-dim2}
All lattice $2$-zonotopes of covering radius greater than 1/2 are unimodularly equivalent to one of the following:
\begin{enumerate}
\item Parallelograms of lattice-width one, generated by $\{(1,0), (0,k)\}$ for some $k\ge 1$.
\item Parallelograms $P_{1,k}$ of lattice-width two and area $k$, $k\ge 2$.
\item Hexagons of lattice-width two, with volume vector $(1,1,k)$ for some $k\ge 1$.
\item The parallelogram $P_{2,5}$ of lattice-width three and volume~$5$.
\end{enumerate}
In particular, Conjecture~\ref{conj:cosimple} holds in dimension two for any number of generators.
\end{thm}

\begin{proof}
Let $Z$ be a lattice $2$-zonotope.
We argue depending on the lattice-width of~$Z$.
Lattice-width one implies that~$Z$ is (up to unimodular equivalence) a parallelogram of the type in part~(1).

Parallelograms of lattice-width two must attain their width either with respect to a diagonal direction or with respect to an edge~$e$. The former are the parallelograms in part (2), and the latter have $\mu(Z) = 1/2$ unless the edge $e$ is primitive, in which case they have area two and are either in part (1) or part (2), with $k=2$.

The non-parallelograms of lattice-width two necessarily have three generators which can be assumed to be $(a,0), (b,1), (c,1)$, as we may apply a unimodular transformation so that the lattice-width is attained in the direction of the second coordinate.
Such a hexagon contains the parallelogram with generators $(a,0)$ and $(b+c,2)$, whose covering radius is $\max\{\frac12,\frac1a\}$, so $\mu(Z) > 1/2$ implies $a=1$ and the volume vector is $(1,1,k)$, with $k=|b-c|$.

Hence, for the rest of the proof we assume that~$Z$ has lattice-width at least three and covering radius $\mu = \mu(Z) > 1/2$, which implies by Corollary~\ref{cor:AverkovWagner}~(ii) that $\vol(Z) \le 8$.
 
We now argue depending on the number of generators in $Z$.
This number must be less than five, since the volume of a lattice zonotope with $n$ generators (in linear general position) is at least $\binom{n}{2}$, and $\binom52 > 8$. So, we have three cases:

\begin{itemize}
\item Suppose that~$Z$ has two generators, that is, it is a parallelogram.
If one of the generators is not primitive, then the fact that this edge has length at least two and that the width of $Z$ with respect to the functional constant on this edge is more than two implies $\mu(Z)\le 1/2$. (This is a particular case of Proposition~\ref{prop:max-bound}, where we project $Z$ along the functional $f$).

If $Z$ is a parallelogram with both edges primitive, Lemma~\ref{lemma:p-q} implies that $Z$ is either in part (2) or it is equivalent to one of $P_{2,5}$, $P_{2,7}$, $P_{3,8}$. Since $P_{2,5}$ is in part (4), we only need to check that 
$P_{2,7}$ and  $P_{3,8}$ have $\mu\le 1/2$. For this:
\begin{itemize}
\item $P_{2,7}$ contains the parallelogram $P'$ generated by $\bs u_1=(2,6)$ and $\bs u_2=(1,1)$; $\bs u_1$ is not primitive and $P'$ has lattice-width two with respect to the functional $f(x,y) = 3x-y$ vanishing on it, so $\mu(P_{2,7}) \le \mu(P') \le 1/2$.
\item $P_{3,8}$ is equivalent to $P_{5,8}$, which contains the parallelogram $P'$ generated by $(2,2)$ and $(4,6)$; since none of them is primitive, $\mu(P_{5,8}) \le \mu(P') \le 1/2$.
\end{itemize}

\item Suppose that~$Z$ has three generators. The volume vector cannot be $(2,2,2)$ or of the form $(a,a,b)$ with $\gcd(a,b)=1$ since that implies lattice-width two. Indeed, without loss of generality assume that the generator separating the two $a$'s (with $a=b=2$ in the first case) is of the form $(p,0)$. Then the other two generators are $(q,h)$ and $(r,h)$ with $a = ph$, so that $b=(q-r)h$. Then, $\gcd(a,b)=1$ implies $h=1$, hence width two. The case $a=b=2$, using $b=(q-r)h$ leads to either $h=1$ (hence width two) or $q-r=1$, implying that one of $q$ or $r$ is even and the corresponding generator is not primitive.

This leaves only the following possible volume vectors $(v_1,v_2,v_3)$ satisfying $v_1+v_2+v_3\le 8$:
\[
(v_1,v_2,v_3) \in \{
(1,2,3), 
(1,2,4), 
(1,2,5), 
(1,3,4),
(2,2,4)
 \} .
\]
That these five zonotopes have $\mu\le \frac12$ is illustrated in Figure~\ref{fig:dimtwo}, where they are shown to contain a parallelepiped with horizontal base of length two and height two. That parallelepiped has $\mu=\frac12$, so a polygon containing it has $\mu$ bounded by that. (For $(1,2,3)$ we show that such a parallelpiped is contained in $Z_1\cup Z_2$ where $Z_1$ and $Z_2$ are translated to one another by the vector $(2,2)\in 2\ZZ^2$, which is enough for the implication.) In fact, their exact covering radii are $\frac12$, $\frac37$, $\frac37$, $\frac37$, and $\frac12$ (see~\cite[Table~2]{cslovjecsekmalikiosisnaszodischymura2022computing} for the first four).
\begin{figure}[htb]
\begin{tabular}{ccccc}
\includegraphics[scale=0.015]{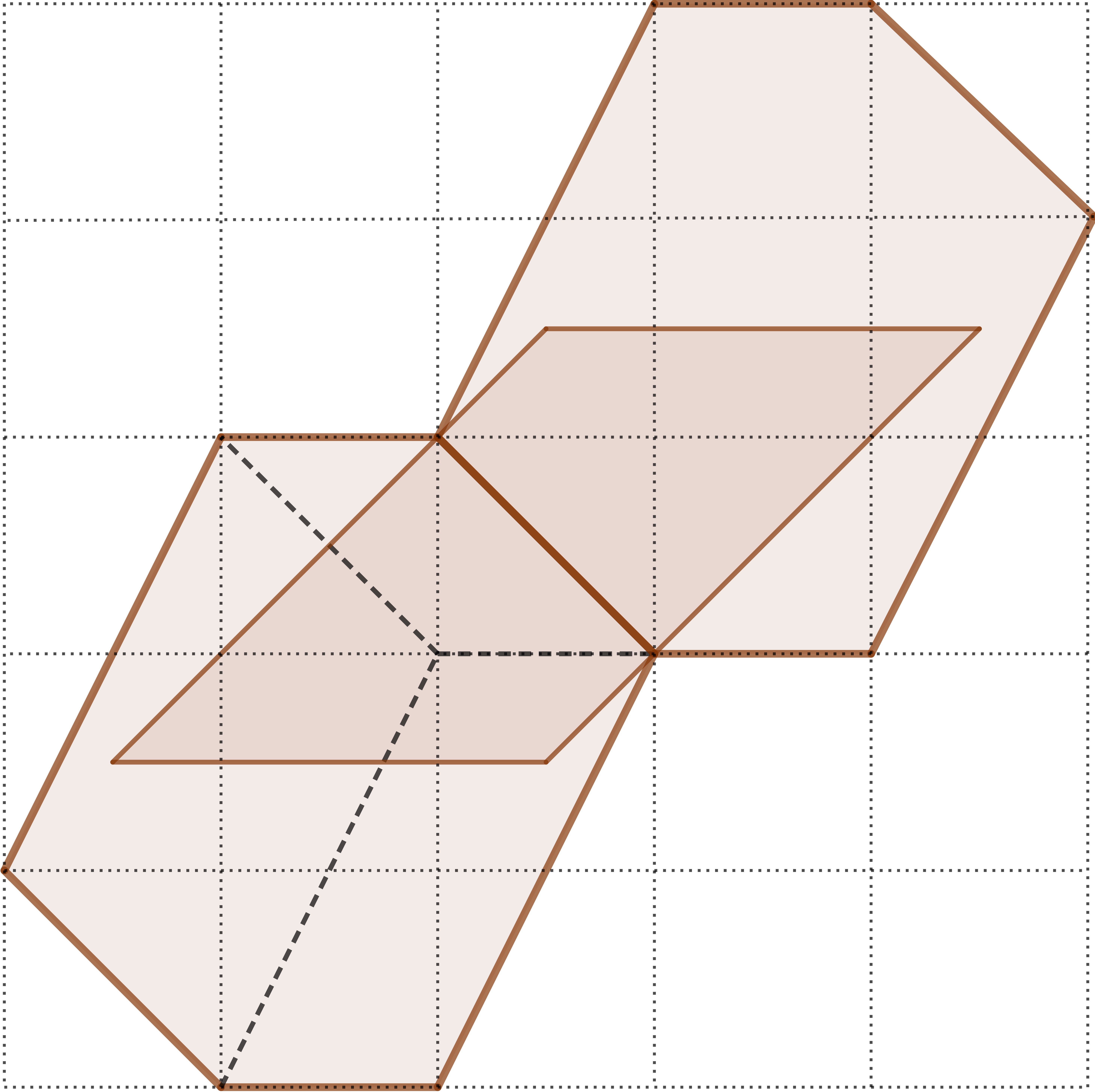}&
\includegraphics[scale=0.015]{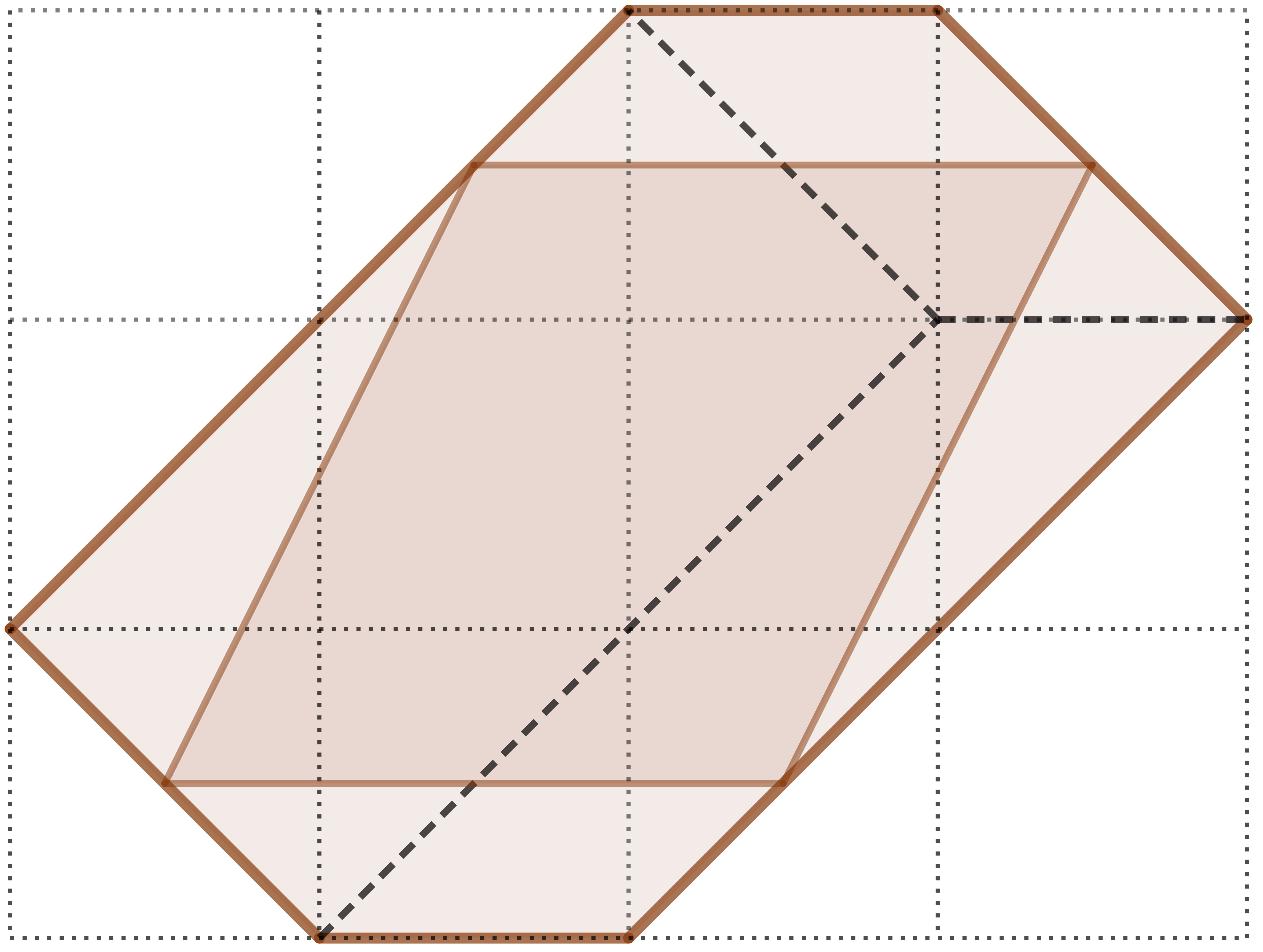}&
\includegraphics[scale=0.015]{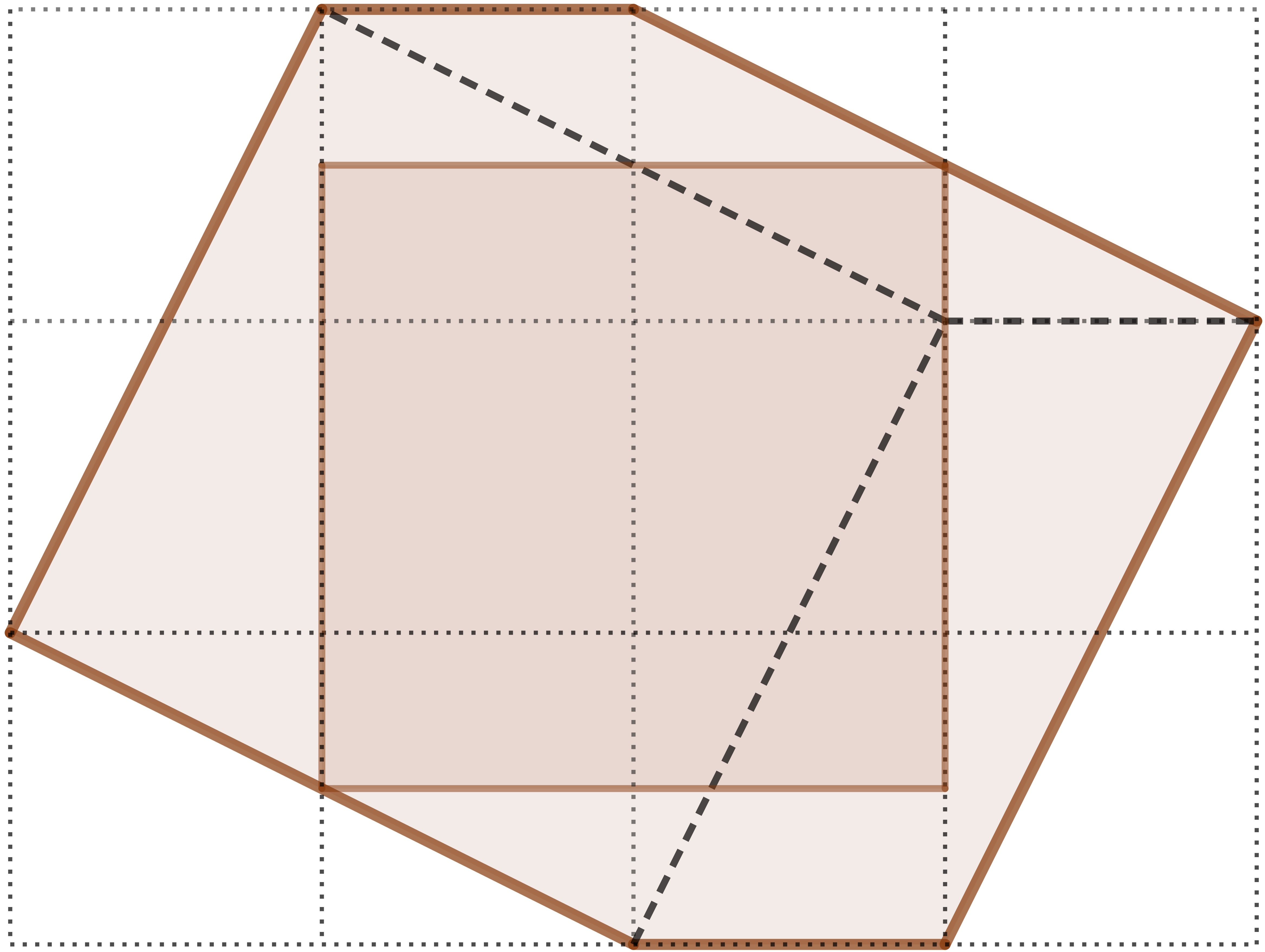}&
\includegraphics[scale=0.015]{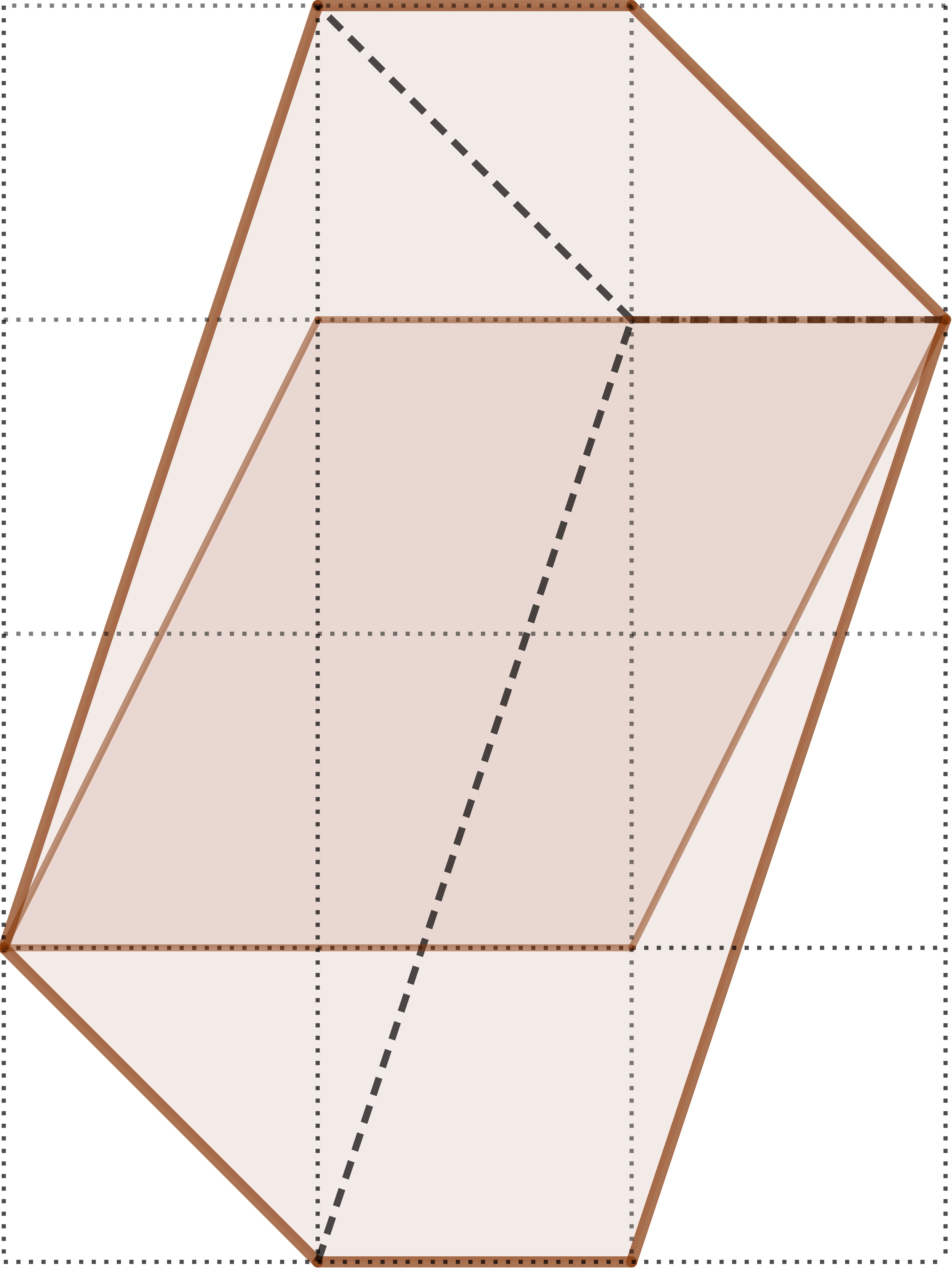}&
\includegraphics[scale=0.015]{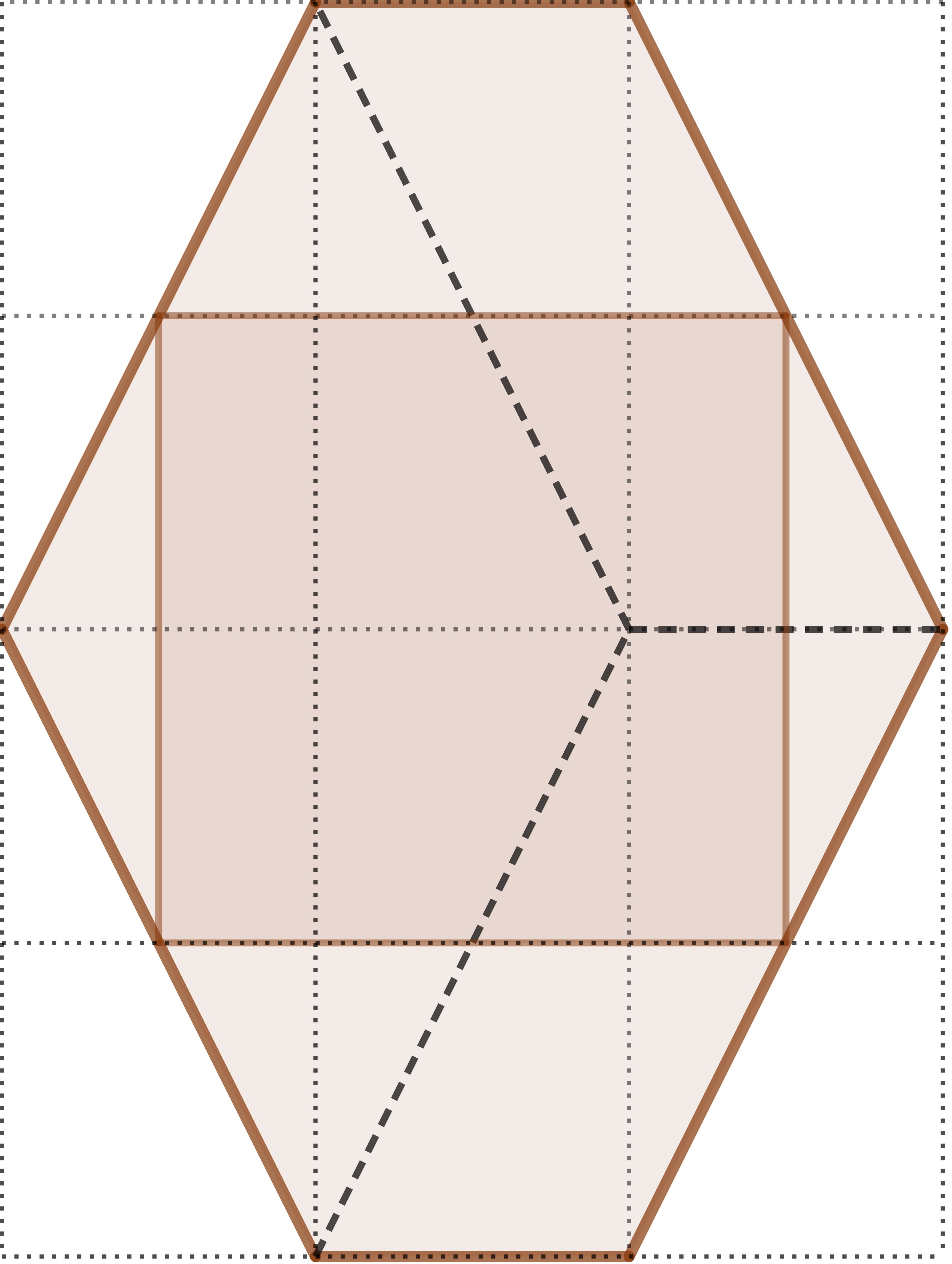}\\
(1,2,3) &(1,2,4)&(1,2,5)&(1,3,4)&(2,2,4)\\
\end{tabular}
\caption{The five hexagons in the proof of Theorem~\ref{thm:cosimple-dim2}, with their volume vectors. In each of them a parallelepiped with base and height equal to 2 is shown, to illustrate that the hexagons have $\mu\le \frac12$}
\label{fig:dimtwo}
\end{figure}

\item If $Z$ has four pairwise non-proportional generators, then the volume of $Z$ equals the sum of the six volumes of the parallelograms generated by the $2$-element subsets of the generators (cf.~\cite[Eq.~(57)]{shephard1974combinatorial}). The only possibilities for the volume $6$-tuples adding up to eight or less are
\[
\{(1,1,1,1,1,1) ,(1,1,1,1,1,2) , (1,1,1,1,1,3), (1,1,1,1,2,2)\}.
\]
These numbers, modulo sign and reordering, are the six $2\times 2$ minors $(p_{ij})_{1 \le i < j \le 4}$ of a $2\times 4$ matrix, hence they have to satisfy 
the following Pl\"ucker relation (see, e.g.~\cite[Section 4.2]{MillerSturmfels}):
\[
p_{14}p_{23} - p_{13}p_{24} + p_{12}p_{34} = 0.
\] 
The only one where this is possible is $(1,1,1,1,1,2)$, corresponding uniquely (modulo unimodular transformation) to the 
zonotope generated by $(1,0)$, $(0,1)$, $(1,1)$, and $(1,-1)$. 
This has covering radius exactly $1/2$ (see Figure~\ref{fig:111112}).
\begin{figure}[htb]
	\begin{center}
\includegraphics[scale=0.015]{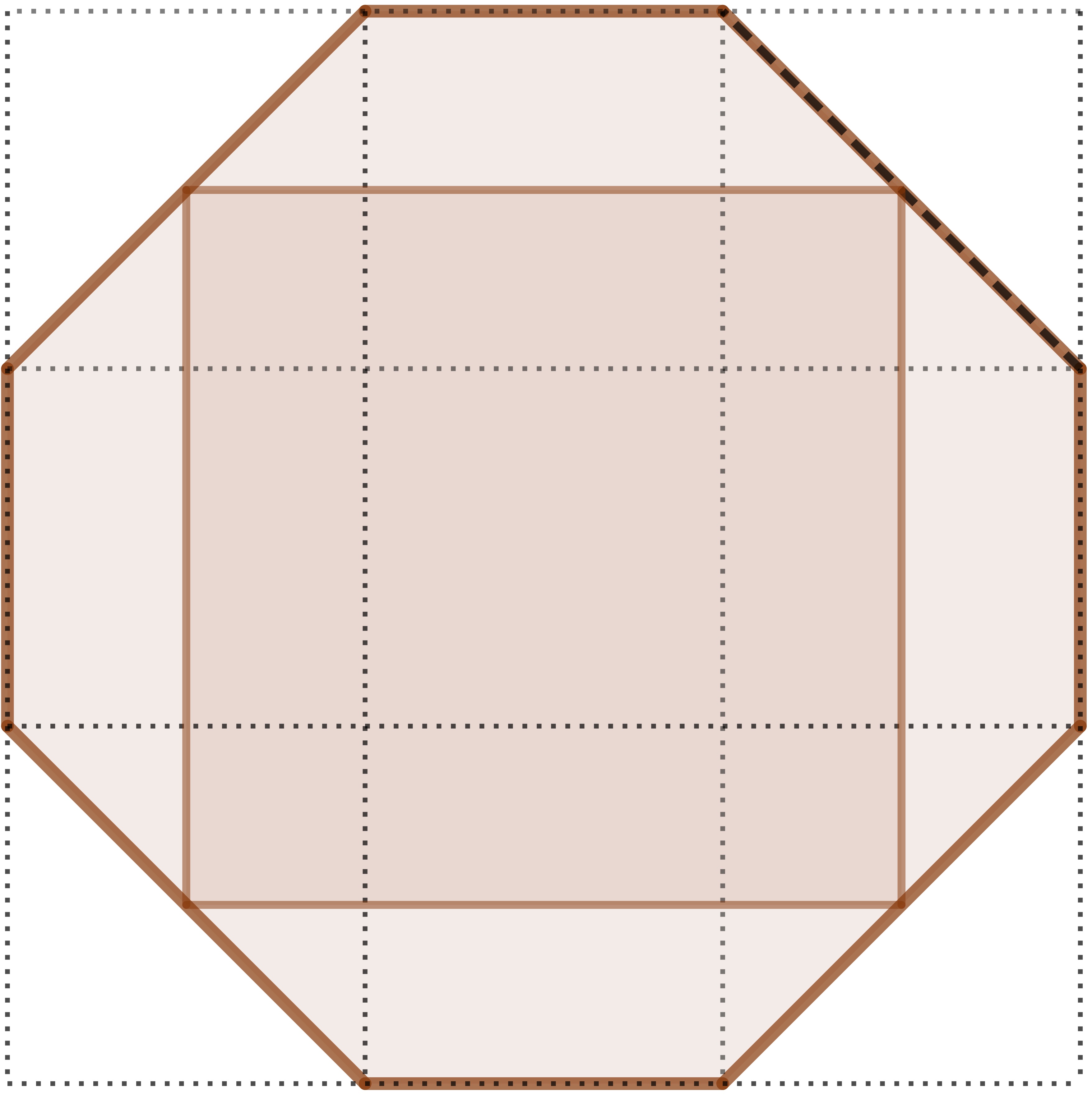}
\end{center}
\caption{The  octagon in the proof of Theorem~\ref{thm:cosimple-dim2}, with an inscribed square implying $\mu\le \frac12$}
\label{fig:111112}
\end{figure}
\qedhere
\end{itemize}
\end{proof}

\begin{rem}
\label{rem:cosimple-dim2-covrads}
The covering radii of the zonotopes in Theorem~\ref{thm:cosimple-dim2} can be computed to be as follows:
\begin{enumerate}
 \item $\mu=1$;
 \item $\mu = 1/2 + 1/k$, if $k$ is even, and $\mu = 1/2 + 1/(2k)$, if $k$ is odd;
 \item $\mu = 1/2 + 1/(2k+2)$, if $k$ is even, and $\mu = 1/2 + 1/(4k+2)$, if $k$ is odd;
 \item $\mu = 3/5$.
\end{enumerate}
In our calculations below (proof of Proposition~\ref{prop:width33}) we need the last one of them, so let us  prove it.
Using Lemma~\ref{lemma:p-q} it is easy to see that $P_{2,5}$ is isomorphic to the parallelogram~$Q$ generated by $(2,-1)$ and $(1,2)$, depicted in the left picture of Figure~\ref{fig:parall5}.
The caption of the figure explains why, indeed, $\mu(P_{2,5}) = \mu(Q) = 3/5$.

\begin{figure}[htb]
	\begin{center}
\includegraphics[scale=0.25]{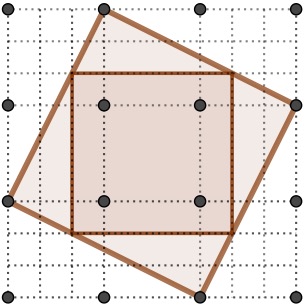}
\qquad\qquad
\raisebox{0.55cm}{\includegraphics[scale=0.15]{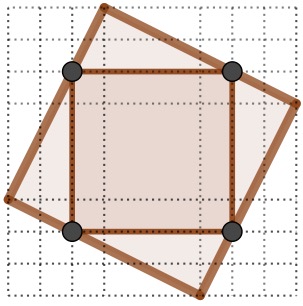}}
\end{center}
\caption{The covering radius of the parallelogram $Q \cong P_{2,5}$ equals~$3/5$: 
The left picture shows that~$Q$ contains a translation of the square $[0,5/3]^2$; hence $\mu(P_{2,5}) \le 3/5$. 
For the equality, consider the right picture, where we scale down~$Q$ by~$3/5$ about its center, so that the axes-parallel square in it becomes a lattice unit square with its vertices in the boundary of~$(3/5)Q$.
Since any smaller dilation will fail to contain points from~$\ZZ^2$, we have that $\mu \cdot P_{2,5} + \ZZ^{2}$ does not cover~$\RR^2$ for any $\mu < 3/5$}
\label{fig:parall5}
\end{figure}
\end{rem}

\begin{prop}\label{prop:allbut2gcd}
	Let $n \geq 4$. Suppose that Conjecture~\ref{slrc} holds for $n-1$ velocities and that we have an instance with $n$ velocities satisfying  $\gcd(v_1,\dots, v_{n})=1$.
	If $\delta = \gcd(v_3,\dots, v_{n})$ satisfies
	\begin{align*}
		\delta \ge 2 + \frac8{n-3} & \quad\text{and $\delta$ is even, or} \\
		\delta \ge 1 + \frac4{n-3} & \quad\text{and $\delta$ is odd,} 
	\end{align*}
	then Conjecture~\ref{slrc} holds for $v_1,\dotsc,v_n$.
\end{prop}

\begin{proof}
In the light of Lemma~\ref{lemma:gcd-geom}, the condition  
$\delta=\gcd(v_3,\dots, v_{n})$ is equivalent to saying that the parallelogram $Z_{\{1,2\}}$ spanned by $\bs u_1$ and~$\bs u_2$ has area $\delta$.

If $\mu(Z_{\{1,2\}}) \le \tfrac{n-1}{n+1}$ then sLRZ holds for $(v_1,\dots,v_n)$ by Proposition~\ref{prop:max-bound} applied to the projection along the plane $\lin(\{\bs u_1, \bs u_2\})$. 
Indeed, every fiber of the projection contains a translated copy of $Z_{\{1,2\}}$, and the image of the projection is an sLRZ with two less generators. 

Hence, we have to consider only the parallelograms with $\mu(Z_{\{1,2\}})>\tfrac{n-1}{n+1} \geq \tfrac35$, listed in Theorem~\ref{thm:cosimple-dim2} and whose covering radii are given in Remark~\ref{rem:cosimple-dim2-covrads}.
Those of part~(1) are discarded by Lemma~\ref{lemma:primitive}, since they have a non-primitive generator amd the one in part~(3) has covering radius $\tfrac35 \leq\tfrac{n-1}{n+1}$, so only those in part~(2) need to be considered. 
These have covering radius  $1/2 +1/\delta$ if their area $\delta$ is even, and $1/2+1/2\delta$, if it is odd. The inequalities
\[
\frac12 + \frac1\delta \leq \frac{n-1}{n+1}
\quad\text{ and }\quad 
\frac12 + \frac1{2\delta} \leq \frac{n-1}{n+1}
\]
are, respectively, equivalent to the ones in the statement.
\end{proof}

The values of $\delta$ not covered by Proposition~\ref{prop:allbut2gcd} quickly decrease with $n$. In fact, 
since~Conjecture~\ref{slrc} holds for $n \leq 4$ (see Remark~\ref{rem:alcantara}), the proposition implies the following:

\begin{cor}
\label{cor:small-gcd}
	Suppose that Conjecture~\ref{slrc} holds for $n-1$ velocities but fails for some vector with $n$ velocities and satisfying $\gcd(v_1,\dots, v_{n})=1$. Then either:
	\begin{enumerate}
	    \item $n\in \{5,6\}$ and $\gcd(v_3,\dots, v_{n}) \in \{1,2,4\}$, or
	    \item $n\ge 7$ and $\gcd(v_3,\dots, v_{n}) \in \{1,2\}$.
	\end{enumerate}
\end{cor}

\section{Dimension three: volume bound for potential counterexamples}
\label{sec:sLRC-dim3}
\label{sec:dim3}

We here prove that any potential counter-example to Conjecture~\ref{conj:cosimple} for dimension $3$ (and hence any potential counterexample to Conjectures~\ref{slrc} and~\ref{slrcgeom} for  $n=4$) has volume bounded by 195.
Independently of the zonotope to be cosimple, we show the following stronger result:

\begin{thm}
\label{thm:dim3-b}
Let $Z$ be a lattice $3$-zonotope with $\mu(Z) \geq 3/5$ and let $w$ be its lattice-width.
Then, $w \leq 6$ and if $w \ge 3$, then either:
\begin{enumerate}
\item $w=3$ and either $\vol(Z) \le 120$ or $Z$ is a parallelepiped projecting to the parallelogram $P_{2,5}$ of Theorem~\ref{thm:cosimple-dim2}.
If $Z$ is an sLRZ then $\vol(Z) \le 80$.
\item $w=4$ and $\vol(Z) \le 195$. 
\item $w=5$ and $\vol(Z) \le 125$. 
\item $w=6$ and $\vol(Z) \le 98$. 
\end{enumerate}
\end{thm}
\noindent That all cosimple zonotopes (hence all sLRZ) have lattice-width at least three follows from Corollary~\ref{cor:widthge3}.
In what follows, we treat separately the cases of lattice-width at least four and equal to three.

\subsection{Zonotopes of lattice-width at least four}\label{subsec:width4+}

To deal with zonotopes of lattice-width at least four, we argue similarly as we did in Corollary~\ref{cor:AverkovWagner} and make use of Lemma~\ref{lemma:IglesiasSantos}, which is the corresponding three-dimensional volume bound for hollow convex bodies.

\goodbreak
\begin{cor}
\label{coro:volume-dim3}
Let $Z$ be a lattice $3$-zonotope of lattice-width $w \ge 4$ and covering radius $\mu \ge 3/5$.
Then, $w \le 6$ and the volume of~$Z$ is upper bounded by 
\begin{enumerate}[(i)]
\item $195$ if $w=4$,
\item $125$ if $w=5$,
\item $98$ if $w=6$.
\end{enumerate}
\end{cor}

\begin{proof}
It is proven in~\cite[Theorem~5.2]{AverkovCMS} that a hollow convex 3-body has lattice-width bounded by $3.972$. 
Hence, for a convex body of covering radius $\mu \ge 3/5$ we have
\[
w\le 3.972 \mu^{-1} \le 3.972 \cdot \frac{5}{3} = 6.62.
\]
This shows $w \le 6$, since the lattice-width of the lattice zonotope~$Z$ is an integer.

For the volume, let $Z'=\mu Z$, which has a hollow translate and lattice-width~$\mu w$.
For $w \ge 5$, we have $\mu w \ge 3$ so we can apply the first bound in Lemma~\ref{lemma:IglesiasSantos}, which gives
\[
\vol(Z) = \mu^{-3} \vol(Z') \le 8 \mu^{-3} \frac{\mu^{3}w^3}{(\mu w-1)^3} = \left(\frac{2w}{\mu w-1}\right)^3 \le \left(\frac{10w}{3w-5}\right)^3.
\]
Plugging in $w=5$ and $w=6$ gives bounds of $(50/10)^3=125$ and $(60/13)^3 = 98.32$, respectively.
The observation that $\vol(Z)$ is an integer gives the bound for cases~(ii) and~(iii).

For $w=4$, we may need to use the first bound of Lemma~\ref{lemma:IglesiasSantos} or the second one, depending on $\mu$, since for $\mu \approx 3/5$ we have that $\mu w \approx 12/5 = 2.4$.
Thus, we use the maximum of the two bounds.
Using $\mu \ge 3/5$, the first bound gives
\[
\vol(Z) = \mu^{-3} \vol(Z') \le  
\mu^{-3} \frac{8 \mu^3w^3}{(\mu w - 1)^3} = 
\frac{8 \cdot 4^3}{(4\mu - 1)^3} \le
\frac{40^3}{7^3} = 186.59,
\]
and the second one gives 
\[
\vol(Z) = \mu^{-3} \vol(Z') \le  
\frac{3 w^3}{4(\mu w - (1+2/\sqrt3))} \le
\frac{48}{12/5 - (1+2/\sqrt3)} = 195.68,
\]
finishing the proof.
\end{proof}

\begin{rem}
The same ideas from the proof of~\cite[Theorem 2.1]{IglesiasSantos}, but assuming~$C$ to be $\bs 0$-symmetric, imply that the first bound in Lemma~\ref{lemma:IglesiasSantos} applies for $w\ge \sqrt{5} \approx 2.236$. Using this extended bound in the proof of Corollary~\ref{coro:volume-dim3} would reduce the volume bound for lattice-width $w=4$ to $186$, instead of $195$.
\end{rem}

\subsection{Zonotopes of lattice-width three}\label{subsec:width3}

For the rest of the section, let $Z$ be a lattice $3$-zonotope of lattice-width three.
We first treat the case where the lattice-width is attained with respect to (at least) two different integer linear functionals, and then we see how the case of a single one splits into two subcases.

\subsubsection{Zonotopes of lattice-width three for two different functionals}

We here assume that $Z$ is a lattice 3-zonotope of lattice-width three with respect to two (linearly independent) functionals $f_1,f_2$. 
Let us first see that there is no loss of generality in assuming that these are the first two coordinates.
Think of $f_1$ and $f_2$ as elements of the dual lattice $(\ZZ^3)^*$. Since
\[
w(Z, \lambda_1 f_1+ \lambda_2 f_2) \le \lambda_1 w(Z,  f_1)  + \lambda_2 w(Z,  f_2), 
\]
there is no loss of generality in assuming that the triangle formed by $f_1$, $f_2$ and the origin contains no other lattice points. Equivalently, that this triangle is unimodular, hence part of a lattice basis (see, for instance, \cite[Chapter~1]{gruberlekkerkerker1987geometry}). Then, a change of basis sends $f_1$ and $f_2$ to the first two coordinates.
 
Then, if we let $\pi:\RR^3\to \RR^2$ be the projection forgetting the third coordinate, we have that $Z':=\pi(Z)$
is a two-dimensional lattice zonotope of lattice-width three that fits in the square $[0,3]^2$.

%

\begin{prop}
\label{prop:width33}
Let $Z$ be a lattice $3$-zonotope of width three attaining its width w.r.t.~two linearly independent functionals, and with $\mu(Z) \ge\frac35$.
Assume that the projection $Z'=\pi(Z)$  is contained in the square $[0,3]^2$ but is different from the parallelogram $P_{2,5}$ of Theorem~\ref{thm:cosimple-dim2}.

Then 
\[
\vol(Z) \le \frac{5\vol(Z')}{3-5\mu(Z')} \le 10\vol(Z').
\]
\end{prop}

\begin{proof}
The width of $Z'$ is also at least three, since a functional giving a certain width to $Z'$ lifts to a functional with the same width on $Z$. Hence, the fact that $Z'$ does not equal $P_{2,5}$ implies,  by Theorem~\ref{thm:cosimple-dim2}, that $\mu(Z') \le 1/2$.

Let $h$ denote the maximum length among the fibers $\{\pi^{-1}( \bs x) \cap Z :\bs x\in Z'\}$, and let $\bs x \in Z'$ be a point attaining this maximum~$h$.
Then, for each $k \in (0,1]$, we find that the zonotope $Z'_k:=\bs x+ k (Z'-\bs x)\subseteq Z'$ has $\mu(Z'_k) = \mu(Z')/k$ and for every $\bs y \in Z'_k$ the length of $\pi^{-1}( \bs y)$ is at least $(1-k) h$.

Taking $k= 5\mu(Z')/3 \leq 5/6$, we have that $\mu(Z'_k) = 3/5$ and that every fiber over a point $\bs y \in Z'_k$ has length at least $\frac{3-5\mu(Z')}{3}h$.
By Proposition~\ref{prop:max-bound}, $\mu(Z) \ge 3/5$ and $\mu(Z') < 3/5$ implies that some $\bs y\in Z'_k$ must have length bounded by $5/3$, that is,
\[
\frac{3-5\mu(Z')}{3}h \le 5/3
\quad\Rightarrow \quad
h \le \frac{5}{3-5\mu(Z')} \leq 10 \,,
\]
also using $\mu(Z') \leq 1/2$.
Since, obviously, $\vol(Z) \le h \vol(Z')$, the result follows.
\end{proof}

\begin{cor}
\label{cor:width3}
Let $Z$ be a lattice $3$-zonotope that has lattice-width three w.r.t.~two linearly independent functionals.
If $\mu(Z) \ge\frac35$, then 
$
\vol(Z) \le 80,
$
unless $Z$ is a parallelepiped projecting to $P_{2,5}$.
\end{cor}

\begin{proof}
If $Z$ projects to $P_{2,5}$, the fact that $P_{2,5}$ is a parallelogram with primitive generators implies that $Z$ is a parallelepiped. Since we assume this does not happen, we can apply the bound of Proposition~\ref{prop:width33}.
Let $Z'=\pi(Z)$, as in that statement.
If $Z'=[0,3]^2$, then $\vol(Z')=9$ and $\mu(Z')=1/3$, so Proposition~\ref{prop:width33} gives $\vol(Z) \le \frac{45}{3-\frac53} = 33.75$.
If $Z' \ne [0,3]^2$, then $\vol(Z') \le8$ and the same result gives $\vol(Z) \le 10 \vol(Z') \leq 80$.
\end{proof}

\subsubsection{Zonotopes of lattice-width three for a unique functional}
\label{sect:width-three-unique-fct}

We first prove some technical lemmas.
In the first one, we say that a polytope~$P$ is \emphd{centrally symmetric} if there is a point $\bs x \in P$ such that $P - \bs x = \bs x - P$.
For example, all zonotopes are centrally symmetric, but the lemma holds without the zonotopal assumption.

\begin{lemma}
\label{lemma:mu-width-vol}
Let $P$ be a centrally symmetric lattice $3$-polytope of lattice-width three with respect to a unique lattice functional $f$.
We assume $f(P)=[0,3]$ and denote $P_k:= P \cap f^{-1}(k)$, for each $k\in [0,3]$.
\begin{enumerate}
\item $\mu(P_k) \le 2 \mu (P_1)$, for every $k\in [2/3,7/3]$.
\item If $P_1$ is a lattice polytope, then 
\[
w_{f}(P) = w(P_1),
\]
where $w_{f}(P)$ denotes the minimum width of $P$ with respect to lattice functionals not proportional to $f$.
\item Assume that $P$ is a zonotope. Then,
 \begin{enumerate}
  \item $\vol(P) \le 3 \vol(P_1)$, and
  \item if, moreover, $P$ has at most one generator~$\bs u$ orthogonal to~$f$, that is, with $f(\bs u)=0$, then $\vol(P) = 2 \vol(P_1)$.
 \end{enumerate}
\end{enumerate}
\end{lemma}

\begin{proof}
Let us start with part~(1):
\begin{itemize}
\item If $k\in [2/3,1]$, then
\[
k P_1 \subseteq (1-k)P_0 + k P_1 \subseteq P_k 
\quad \Rightarrow \quad
\mu(P_k) \le \frac1{k}\mu(P_1) \le \frac32 \mu(P_1) < 2 \mu(P_1).
\]
\item If $k\in [1, 3/2]$, then 
\[
(2-k) P_1 \subseteq (2-k)P_1 + (k-1) P_2 \subseteq P_k 
\quad \Rightarrow \quad
\mu(P_k) \le \frac1{2-k}\mu(P_1) \le 2 \mu(P_1),
\]
where equality can possibly hold only for $k = 3/2$.
\item If $k \in [3/2,7/3]$, then by central symmetry of~$P$ around a point~$\bs x \in P$ with $f(\bs x) = 3/2$, we get $\mu(P_k) = \mu(P_{3-k}) \leq 2 \mu(P_1)$.
\end{itemize}
Let us now prove part~(2).
That $w_{f}(P) \ge w(P_1)$ is clear: every lattice functional~$f'$ not proportional to~$f$ restricts to a non-zero lattice functional on $P_1$, so the width of $P_1$ is smaller than or equal to the width of~$P$ with respect to~$f'$.

For the converse, let us assume without loss of generality that~$f$ equals the third coordinate, so that we identify each $P_k\subseteq \RR^2\times\{k\}$ with its projection along that coordinate.

Let $g: \RR^2 \to \RR$ be a lattice functional attaining the lattice-width of $P_1$.
By central symmetry of~$P$ and the hypothesis on~$P_1$ being a lattice polytope, $g(P_1)$ and~$g(P_2)$ are integer segments of the same length, equal to $w := w(P_1)$.
Let $m$ be the integer with
\[
g(P_2) = m + g(P_1)
\]
and consider the functional 
\[
f'(x_1,x_2,x_3) := g(x_1,x_2)  - m x_3.
\]
Let $S$ be the segment $S:=g(P_1) - m$.
By construction, $f'(P_1)= f'(P_2)=S$; and we only need to prove $f'(P)\subseteq S$ and get $w_{f}(P) \le \length(f'(P)) \le \length(S) = w(P_1)$.

{\bf Claim:} \emph{$f'(P_1)= f'(P_2)=S$ implies $f'(P)\subseteq S$.}

\noindent Consider the 2-dimensional image $Q:= F(P)$ under the map
\[
F: \RR^3 \to \RR^2 \quad \textrm{ with } \quad \bs p \mapsto (f'(\bs p), f(\bs p)).
\]
$Q$ is a 2-dimensional lattice polytope contained in $\RR\times [0,3]$ and with $Q_1= Q_2 =S$, where we define $Q_k := Q \cap f^{-1}(k)$ analogously to~$P_k$.
This implies $Q \subseteq S \times [0,3]$, that is, $f'(P) \subseteq S$, as desired.

To prove part (3) we first claim a similar result in dimension two:

{\bf Claim:} \emph{Let $P'$ be a lattice $2$-zonotope of lattice-width three with respect to a lattice functional~$\bar f$, with $\bar f(P')=[0,3]$ and denote $P'_k:= P' \cap {\bar f}^{-1}(k)$.
Then, $\vol(P') \le 3 \vol(P'_1)$, and if, moreover, $P'_0$ is a single point, then $\vol(P') = 2 \vol(P'_1)$}.

\noindent To prove it, consider the parallelogram $P'':=P'\cap {\bar f}^{-1}([1,2])$. $P'$ is contained in the parallelogram obtained by extending~$P''$ along the edges not orthogonal to~$f'$, so $\vol(P') \le 3 \vol(P'') = 3\vol(P'_1)$.
If $P'_0$ (and hence $P'_3$) is a single point, then $P' \setminus P''$ is the union of two triangles of area~$\frac12\vol(P'')$, thus $\vol(P') = 2 \vol(P'') = 2\vol(P'_1)$.

We now use induction on the number~$m$ of generators of the zonotope~$P$ orthogonal to~$f$, with base case $m=0$.
Since the (absolute) values of $f$ on generators \emph{not} orthogonal to~$f$ add up to three, $m=0$ implies that~$P$ has only three generators and~$f$ takes value~$1$ in each of them.
The volume of~$P$ is then the determinant of the three generators and the area of the triangle~$P_1$ is half the determinant, so statement~(b) holds in this case.

For the induction step, let $\bs u$ be a generator orthogonal to~$f$.
Let~$Q$ be the (perhaps two-dimensional) zonotope generated by the remaining  generators of~$P$, and let~$P'$ be the projection of~$P$ along the direction of~$\bs u$.
Then,
\begin{align}
\vol(P) &= \vol(Q) + \ell(\bs u) \vol(P'),\label{eqn:auxlemma-1}
\end{align}
where $\ell(\bs u)$ denotes lattice length (see Section~\ref{sec:projections-lrz}).
Similarly, with the obvious notations,
\begin{align}
\vol(P_1) &= \vol(Q_1) + \ell(\bs u) \vol(P'_1).\label{eqn:auxlemma-2}
\end{align}
Notice that, if~$Q$ is two-dimensional, then $\vol(Q)=\vol(Q_1)=0$.
Now, if $\bs u$ is the only generator orthogonal to~$f$, then~$P'_0$ is a single point, so that $\vol(P') = 2 \vol(P'_1)$ by the claim, and the induction hypothesis is the case $m=0$ implying $\vol(Q) = 2 \vol(Q_1)$.
Combining this with the identities~\eqref{eqn:auxlemma-1} and~\eqref{eqn:auxlemma-2} proves statement~(b).
If $m \geq 2$, we have $\vol(P') \le 3 \vol(P'_1)$ by the claim and $\vol(Q) \leq 3 \vol(Q_1)$ by induction hypothesis.
Again combining this with~\eqref{eqn:auxlemma-1} and~\eqref{eqn:auxlemma-2} then proves statement~(a).
\end{proof}

%

Now, let $Z$ be a lattice $3$-zonotope of lattice-width three attained by a unique functional~$f$.
As before, we denote $Z_k = Z \cap f^{-1}(k)$ below, for $k \in \RR$.
Since the width of~$Z$ for the functional~$f$ equals the sum of the (absolute) values of~$f$ on the generators, lattice-width three implies one of the following possibilities for the generators of~$Z$ that are not orthogonal to~$f$:

\begin{enumerate}
\item There are three of them, and $f$ takes value $1$ on each of the three.
\item There are two of them, and $f$ takes values $1$ and $2$ on them, respectively.
\item There is a single one, and $f$ takes value $3$ on it.
\end{enumerate}
\noindent The last case is easy to discard:

\begin{prop}[Case (3) for $f$]
If $Z$ is a lattice 3-zonotope of lattice-width three with respect to a unique functional~$f$ and has exactly one generator not orthogonal to~$f$, then $\mu(Z) < 3/5$.
\end{prop}

\begin{proof}
Assume that the linear functional~$f$ is the third coordinate and hence the non-orthogonal generator is of the form $\bs u=(p,q,3)$. 
By applying the unimodular transformation
\[
\begin{pmatrix}
x_1\\x_2\\x_3 
\end{pmatrix}
\mapsto 
\begin{pmatrix}
x_1-\floor{{p}/3} x_3 \\ x_2-\floor{{q}/3} x_3 \\ x_3
\end{pmatrix}
\]
there is no loss of generality in assuming that $p,q\in \{0,1,2\}$.

To seek a contradiction suppose that $\mu(Z) \ge 3/5$.
Since $Z=Z_0 + \ent{\bs 0,\bs u}$, Proposition~\ref{prop:max-bound} gives
\[
\frac35 \le \mu(Z) \le \max\left\{\mu(Z_0),\frac13\right\},
\]
implying that $\mu(Z_0) >3/5$.
Then, Theorem~\ref{thm:cosimple-dim2} implies that~$Z_0$ has lattice-width one or two, since $Z_0 =P_{2,5}$ would imply $Z$ to have width two with respect to two different functionals.

We can assume the lattice-width $w\le 2$ of $Z_0$ to be attained with respect to the first coordinate, that is, $Z_0 \subseteq [0,w]\times \RR\times \{0\}$.
Then, the lattice-width of~$Z$ with respect to the first coordinate is $w+p$, which is at most $3$ except if $w=p=2$.
But in that case~$Z$ has width three with respect to $f'(\bs x) = x_1-x_3$, since $f'(Z_0)=[0,2]$ and $f'(Z) = f'(Z_0 + [\bs 0,\bs u]) =[-1,2]$.
The contradiction is that in both cases we have a second functional giving~$Z$ width three.
\end{proof}

Case (1) is also easy, since in this case~$Z_1$ is a lattice polytope and we can readily apply to~$Z$ the three parts of Lemma~\ref{lemma:mu-width-vol}.

\begin{prop}[Case (1) for $f$]
\label{prop:width3.case1}
Let $Z$ be a lattice $3$-zonotope of lattice-width three with respect to a single functional~$f$ and with $\mu(Z) \ge 3/5$.
If $Z$ has three generators not orthogonal to~$f$, then
\[
\vol(Z_1) \le 40.
\]
\end{prop}

\begin{proof}
By Proposition~\ref{prop:max-bound}, $\mu(Z) \ge 3/5$ implies that there is a $k\in [2/3, 7/3]$ with $\mu(Z_k) \ge 3/5$. Part (1) of Lemma~\ref{lemma:mu-width-vol} then gives $\mu(Z_1) \ge 3/10$. 

Since $f$ is the only functional giving width three to $Z$ and $Z_1$ is a lattice polytope, part (2) of Lemma~\ref{lemma:mu-width-vol} says that~$Z_1$ has lattice-width at least four, so  Corollary~\ref{cor:AverkovWagner} with $\mu w \ge 6/5$ implies
\[
\vol\left(Z_1\right) \le \frac{16}{2/5} = 40,
\]
as claimed.
\end{proof}

\begin{cor}[Case (1) for $f$]
\label{cor:width3.case1}
Let $Z$ be a lattice $3$-zonotope of lattice-width three attained by a unique functional~$f$ and with $\mu(Z)\ge 3/5$.
If $Z$ has three generators not orthogonal to $f$, then
$ \vol(Z) \le 120$.
If, moreover, $Z$ is an sLRZ, then $\vol(Z) \le 80$.
\end{cor}

\begin{proof}
The  bounds follow from Proposition~\ref{prop:width3.case1} and part~(3) of Lemma~\ref{lemma:mu-width-vol}.
In the second bound we use that an sLRZ of dimension three has four generators.
Since three of them are not orthogonal to~$f$ only one can be orthogonal. 
\end{proof}

For Case (2) we have a stronger form of part~(1) and a variation of part~(2) of Lemma~\ref{lemma:mu-width-vol}, since $Z_1$ may not be a lattice polytope.

\begin{lemma}
\label{lemma:width3.case2}
Let $P$ be a lattice $3$-zonotope of lattice-width three with respect to a unique functional~$f$, and suppose that it has two generators $\bs u_1$ and $\bs u_2$ with $f(\bs u_1)=1$ and $f(\bs u_2)=2$.
Assume further that $f(P) = [0,3]$.
Then, with the notations of Lemma~\ref{lemma:mu-width-vol}, it holds
\begin{enumerate}
 \item $\mu(P_k) \le \frac32 \mu (P_1)$, for every $k\in [2/3,7/3]$, and
 \item 
\[
w_{f}(P) = \lceil w(P_1)\rceil \le w(P_1) + \frac12.
\]
\end{enumerate}
\end{lemma}

\begin{proof}
To make things concrete, assume without loss of generality that $f$ equals the third coordinate  and let
\[
\bs u_1=(\bs p,1), \quad \bs u_2=(\bs q,2),
\]
for some $\bs p, \bs q \in \ZZ^2$.
Calling $T\subseteq \RR^2$ the segment with endpoints~$\bs p$ and~$\frac12\bs q$, we have that
\[
P_1 = P_0 + (T \times \{1\})
\quad \text{ and }\quad
P_2 =  \frac12 \bs q + P_1,
\]
where the last equality uses that $P_1$ is a (perhaps non-lattice) zonotope, hence centrally symmetric.
In fact, for all $k\in [1,2]$, we now have that $P_k$ is a translation of~$P_1$ by the vector $\frac{k-1}2 \bs q$.
This implies that $\mu(P_k) = \mu(P_1)$ for $k\in [1,2]$, and  the argument  in the proof of Lemma~\ref{lemma:mu-width-vol} gave $\mu(P_k) \le \frac32 \mu (P_1)$ for $k\in [2/3, 1) \cup (2,7/3]$. 
This proves part (1).

For part (2), as in the proof of part (2) of Lemma~\ref{lemma:mu-width-vol} we have that $w_{f}(P) \ge w(P_1)$ is obvious, and for the converse we let  $g: \RR^2 \to \RR$ be a lattice functional attaining the lattice-width of~$P_1$.
If $g(P_1)$ (hence $g(P_2)$, by central symmetry) are lattice segments then all we said in the proof of Lemma~\ref{lemma:mu-width-vol} remains valid, and we get
\[
w_{f}(P) = w(P_1).
\]
So, suppose that $g(P_1)$ is not a lattice segment.
Since one endpoint of~$T$ is a lattice point and the other is half-integral, we have that one endpoint of $g(P_1)$ is an integer and the other a half-integer.
Without loss of generality assume $g(P_1)=[a,b-1/2]$, with $a < b$.
Let $\bs p_1, \bs p_2 \in \ZZ^2$ be lattice points with $g(\bs p_1) = g(\bs p_2) =b$ and that lie sufficiently far from each other in opposite directions on the line $f^{-1}(1) \cap (g^{-1}(b) \times \{1\})$.
This implies that
\[
P'_1:= \conv\left(P_1 \cup \{(\bs p_1,1), (\bs p_2,1)\}\right)
\]
is a lattice polytope containing $P_1$, and 
\[
P':= \conv(P \cup \{(\bs p_1,1), (\bs p_2,1) , (\bs x- \bs p_1,2), (\bs x - \bs p_2,2) \}),
\]
where $\frac12\bs x$ is the (half-integral) center of~$P$, is a centrally symmetric lattice polytope containing $P$ and with $P'\cap f^{-1}(1) = P'_1$.
Also, by construction, $g(P'_1)= [a,b]$ so its width equals
\[
b-a =\lceil w(P_1)\rceil = w(P_1) + \frac12.
\]
The result follows from applying part (2) of Lemma~\ref{lemma:mu-width-vol} to $P'$.
\end{proof}

\begin{cor}[Case (2) for $f$]
\label{prop:width3.case2}
Let $Z$ be a lattice $3$-zonotope of lattice-width three with respect to a single functional~$f$ and with $\mu(Z) \geq 3/5$.
If $Z$ has exactly two generators not orthogonal to $f$, then
\[
\vol(Z_1) \le \frac{245}{16} ,
\]
hence
\[
\vol(Z) \le  3\cdot \frac{245}{16} <  46.
\]
\end{cor}

\begin{proof}
The second inequality follows from the first one by part~(3) of Lemma~\ref{lemma:mu-width-vol}. 

For the first inequality we simply modify the proof of Proposition~\ref{prop:width3.case1} as indicated by Lemma~\ref{lemma:width3.case2}:

By Proposition~\ref{prop:max-bound}, $\mu(Z) \ge 3/5$ implies that there is a $k\in [2/3, 7/3]$ with $\mu(Z_k) \ge3/5$. 
Part (1) of Lemma~\ref{lemma:width3.case2} then gives $\mu(Z_1) \ge 2/5$. 

Since $f$ is the only functional giving width three to $Z$, part (2) of Lemma~\ref{lemma:width3.case2} says that 
$Z_1$ has width at least $7/2$, so  Corollary~\ref{cor:AverkovWagner} with $\mu w \ge 7/5$ implies
\[
\vol\left(Z_1\right) \le \frac{49/4}{4/5} = \frac{245}{16},
\]
as claimed.
\end{proof}


\bibliographystyle{amsplain}
\bibliography{mybib}
		
\end{document}